\documentclass[12pt,reqno]{amsart}
\usepackage{amsmath,amsthm,amssymb,amsfonts,amscd}
\usepackage{mathrsfs}
\usepackage{bbm}
\usepackage{bbding}
\usepackage{hyperref}
\usepackage{geometry}
\usepackage{color}
\usepackage{xcolor}

\usepackage{picture,epic}
\usepackage{tikz}

\usepackage{enumitem}

\newcommand{\bea}{\begin{eqnarray}}
\newcommand{\eea}{\end{eqnarray}}
\newcommand{\bna}{\begin{eqnarray*}}
\newcommand{\ena}{\end{eqnarray*}}

\numberwithin{equation}{section}

\theoremstyle{plain}
\newtheorem{theorem}{Theorem}[section]
\newtheorem{lemma}[theorem]{Lemma}

\theoremstyle{definition}

\theoremstyle{remark}
\newtheorem{remark}[theorem]{Remark}

\renewcommand{\Re}{\operatorname{Re}}

\newcommand{\sgn}{\operatorname{sgn}}

\newcommand{\supp}{\operatorname{supp}}

\newcommand{\GL}{\operatorname{GL}}

\renewcommand{\mod}{\operatorname{mod}\ }

\newcommand{\dd}{\mathrm{d}}

\makeatletter
\def\@tocline#1#2#3#4#5#6#7{\relax
  \ifnum #1>\c@tocdepth 
  \else
    \par \addpenalty\@secpenalty\addvspace{#2}%
    \begingroup \hyphenpenalty\@M
    \@ifempty{#4}{%
      \@tempdima\csname r@tocindent\number#1\endcsname\relax
    }{%
      \@tempdima#4\relax
    }%
    \parindent\z@ \leftskip#3\relax \advance\leftskip\@tempdima\relax
    \rightskip\@pnumwidth plus4em \parfillskip-\@pnumwidth
    #5\leavevmode\hskip-\@tempdima
      \ifcase #1
       \or\or \hskip 1em \or \hskip 2em \else \hskip 3em \fi%
      #6\nobreak\relax
    \hfill\hbox to\@pnumwidth{\@tocpagenum{#7}}\par
    \nobreak
    \endgroup
  \fi}
\makeatother

\begin{document}
	
\title{Hybrid subconvexity bounds for twists of $\rm GL_2\times\rm GL_2$ $L$-functions}
		
\author{Chenchen Shao}
\address{School of Mathematics, Shandong University \\ Jinan,
Shandong 250100, China}
\email{scc\_sdu@mail.sdu.edu.cn}

\author{Huimin Zhang}
\address{Data Science Institute and School of Mathematics, Shandong University \\
Jinan, Shandong 250100, China}
\email{hmzhang@mail.sdu.edu.cn}

\date{}

\begin{abstract}
In this paper, we prove hybrid subconvexity bounds for $\rm GL_2\times \rm GL_2$ Rankin--Selberg $L$-functions twisted by a primitive Dirichlet character $\chi$ modulo a prime power,
in the $t$ and depth aspects.
\end{abstract}

\thanks{This work is in part supported by the National Key R$\&$D Program of China (No. 2021YFA1000700). H. Zhang is also supported by NSFC (No. 12031008) and China Scholarship Council (No. 202206220071).}
	
	
	\maketitle
	
\section{Introduction}\label{introduction}
The subconvexity problem for general $L$-functions is one of the central questions in analytic number theory and it has many applications in number theory.
Let $f$ and $g$ be holomorphic or Maass cusp forms for $\rm SL_2(\mathbb{Z})$ with normalized
Fourier coefficients $\lambda_f(n)$ and $\lambda_g(n)$ respectively. Let $\chi$ be a primitive Dirichlet character of prime power conductor $p^{\kappa}$ with $p$ prime.
The $\rm{GL_2}\times \rm{GL_2}\times \rm{GL_1}$ Rankin--Selberg $L$-function 
is defined by
$$
L\left(s,f\times g \times \chi\right)= L(2s, \chi^2)\sum_{n=1}^{\infty}\frac{\lambda_f(n)\lambda_g(n)\chi(n)}{n^s},
$$
for $\Re(s)>1$.
And this $L$-function has an analytic continuation to the whole complex plane.
The subconvexity problem for $\rm{GL_2}\times \rm{GL_2}\times \rm{GL_1}$ $L$-functions is to obtain a bound of the form
$$
L\left(\frac{1}{2}+i t, f \times g \times \chi\right) \ll (p^{\kappa}(|t|+1))^{1-\delta+\varepsilon}
$$
with $0 \leq \delta \leq 1$.

The convexity bounds correspond to the bounds with $\delta=0$, which follow from the functional equation and the Phragm\'{e}n--Lindel\"{o}f convexity principle. The Lindel\"{o}f hypothesis is the statement with $\delta=1$.

The literature on the subconvexity problem is rich and the first result can be traced back over 100 years ago. For the $\rm GL_1$ case, Weyl \cite{Weyl} proved a subconvexity bound for Riemann zeta function. Burgess \cite{Burgess} proved a subconvexity bound for Dirichlet $L$-functions. For $\rm GL_2$ case, see Michel--Venkatesh \cite{MichelVenkatesh} and the references therein. Many progress on $\rm GL_3$ $L$-functions has been made in the past ten years (see \cite{Blomer}, \cite{Li}, \cite{Munshi2015circleII}, \cite{Munshi2015circleIII}, \cite{HuangXu}, etc.).

Subconvexity for $\rm GL_2\times\rm GL_2$ $L$-functions have been solved by many number theorists. For the $t$-aspect, see Bernstein--Reznikov \cite{BernsteinReznikov}, Acharya--Sharma--Singh \cite{ASS},
Blomer--Jana--Nelson \cite{BlomerJanaNelson}, Huang--Sun--Zhang \cite{HuangSunZhang}. For the depth aspect, see Sun \cite{Sun}.

In this paper, we extend the techniques to prove hybrid subconvexity bounds for twists of $\rm GL_2\times\rm GL_2$ $L$-functions, which generalizes the best known bounds using delta method in the depth aspect and in the $t$-aspect simultaneously.

\begin{theorem}\label{main-theorem}
  Let $f$ and $g$ be holomorphic or Maass cusp forms for $\rm SL_2(\mathbb{Z})$ and let $\chi$ be a primitive Dirichlet character of prime power conductor $p^{\kappa}$ with p an odd prime and $\kappa >12$.
  The following hybrid subconvexity bound holds:
  $$
  L\left(\frac{1}{2}+it,f\times g \times \chi\right) \ll_{\varepsilon} p^{\frac{3}{4}}(p^{\kappa})^{1-\frac{1}{16}+\varepsilon}(|t|+1)^{1-\frac{1}{10}+\varepsilon}.
  $$
\end{theorem}

{\bf Acknowledgements.}
The authors would like to thank Professors Valentin Blomer, Bingrong Huang and Lilu Zhao for many helpful comments and encouragement.
The authors also thank Dr. Xinchen Miao for pointing out some references. A large part of this research is finished when Huimin Zhang is visiting the University of Bonn. She is grateful to the China Scholarship Council for supporting her studies at the University of Bonn, and also thanks Mathematical Institute of the University of Bonn for their hospitality.

\medskip

\bigskip
\noindent
{\bf Notation.}
Throughout the paper, the letters $\varepsilon$ and $A$ denote arbitrarily small and large
positive constants, respectively, not necessarily the same at different occurrences.
We use $A\asymp B$ to mean that $c_1B\leq |A|\leq c_2B$ for some positive constants $c_1$ and $c_2$. The symbol $\ll_{a,b,c}$ denotes that the implied constant depends at most on $a$, $b$ and $c$.
As usual, $e(x)=e^{2 \pi i x}$ and
$q\sim R$ means $R<q\leq 2R$.

\section{Outline of the proof}\label{outline-of-proof}
We give a brief sketch of the proof of Theorem \ref{main-theorem}. In section \ref{AFE}, after using the approximate functional equation,
we are left with estimating the sum
\bna
S(N)=\sum_{n\sim N}\lambda_f(n)\lambda_g(n)n^{-it}\chi(n)
\ena
with $N\ll p^{2\kappa}t^2$. In section \ref{subsec:DFIdelta}, we apply the conductor lowering mechanism for the depth aspect, and apply a version of the delta method of Duke, Friedlander and Iwaniec to decouple the oscillations. For simplicity we consider the generic case, i.e. $N=p^{2\kappa}t^2$ and $q\sim Q=\sqrt{N/p^{\lambda}K}$ with $(q,p)=1$, then we essentially have
\begin{equation*}
\begin{aligned}
S(N)\approx \int_{u\sim 1}\frac{1}{Qp^{\lambda}}\sum_{q\sim Q\atop (q,p)=1}\frac{1}{q}\;
   \sideset{}{^{\star}}\sum\limits_{a(\text{mod} \,qp^{\lambda})}&\sum_{n\sim N}
   \lambda_{g}(n)e\left(-\frac{na}{qp^{\lambda}}\right)e\left(-\frac{nu}{Qqp^{\lambda}}\right)\\
   &\cdot\sum_{m\sim N}\lambda_{f}(m)m^{-it}\chi(m)e\left(\frac{ma}{qp^{\lambda}}\right)e\left(\frac{mu}{Qqp^{\lambda}}\right)\mathrm{d}u.
\end{aligned}
\end{equation*}
Trivially we have $S(N)\ll N^2$. Note that we don't need the conductor lowering trick for the $t$-aspect as \cite{ASS}, in fact, the $u$-integral above plays the same role.

Next we apply the $\rm GL_2$ Voronoi summation formula on the sums over $n$ and $m$ in section \ref{Voronoi}.
For the $n$-sum, we roughly arrive at
\begin{equation*}
\begin{aligned}
S_1=\sum_{n\sim N}\lambda_{g}(n)e\left(-\frac{na}{qp^{\lambda}}\right)e\left(-\frac{nu}{Qqp^{\lambda}}\right)\approx \frac{N}{qp^{\lambda}}\sum_{n=1}^{\infty}\lambda_g(n)e\left(\frac{\overline{a}n}{qp^{\lambda}}\right)
\Phi_1\left(\frac{nN}{q^2p^{2\lambda}}\right),
\end{aligned}
\end{equation*}
for certain weight function $\Phi_1$. The conductor is $K^2Q^2p^{2\lambda}$, so the dual length becomes $n\asymp \frac{K^2Q^2p^{2\lambda}}{N}=Kp^{\lambda}$.
By Lemma \ref{Phi 1 analysis}, we can trivially estimate the dual sum as $\frac{N}{Qp^{\lambda}}\cdot Kp^{\lambda}\cdot(\frac{Kp^{\lambda}N}{Q^2p^{2\lambda}})^{-1/2}$. So we save $\frac{N^{1/2}}{p^{\lambda/2}K^{1/2}}$ from $n$-sum.
By the stationary phase method, we save $K^{1/2}$ from the $u$-integral.
And for the $m$-sum, we get
\begin{equation*}
\begin{aligned}
S_2\approx \frac{N}{qp^{\kappa}}\sum_{m}\lambda_f(m)e\left(\frac{-\overline{a^{\ast}}m}{qp^{\kappa}}\right)
\Phi_2\left(\frac{mN}{q^2p^{2\kappa}}\right),
\end{aligned}
\end{equation*}
for certain weight function $\Phi_2$. The conductor is $t^2Q^2p^{2\kappa}$, and hence the dual length becomes $m\asymp \frac{t^2Q^2p^{2\kappa}}{N}=\frac{t^2}{K}p^{2\kappa-\lambda}$. By Lemma \ref{lemma:voronoiGL2-Maass-asymptotic}, we can trivially estimate the dual sum as $\frac{N}{qp^{\kappa}}\cdot (\frac{N}{q^2p^{2\kappa}})^{-1/4}\cdot(\frac{t^2Q^2p^{2\kappa}}{N})^{3/4}\cdot t^{-1/2}$.
So we save $\frac{N^{1/2}p^{\lambda/2}K^{1/2}}{p^{\kappa}t}$ from $m$-sum.

Because of a Ramanujan sum appearing in the character sum and assuming
square-root cancellation for the remaining part of the character sum,
until now we have saved
\bna
\frac{N^{1/2}}{p^{\lambda/2}K^{1/2}}\times K^{1/2}\times\frac{N^{1/2}p^{\lambda/2}K^{1/2}}{p^{\kappa}t}\times Qp^{\lambda/2}
\sim N.
\ena
So we are on the convex bound and need to save a little more.

After an application of the Cauchy--Schwarz inequality and the Poisson summation formula to the sum over $m$ in Section \ref{CauchyPoisson}, we arrive at
\begin{equation*}
\begin{aligned}
S(N)\ll \frac{Q}{p^{(\kappa+\lambda)/2}}\sum_{\substack{q\sim R\\ (q,p)=1}} \left(
  \sum_{n_1\asymp \frac{N}{Q^2} } |\lambda_g(n_1)|^2\sum_{\substack{n_2\asymp \frac{N}{Q^2} \\ n_2\equiv n_1 \bmod q}}\frac{R^2p^{\kappa}t^2}{Nq} \sum_{m\in\mathbb{Z}} \mathfrak{C}(m,n_1,n_2)\cdot \mathfrak{I}(m,n_1,n_2)\right)^{1/2}
\end{aligned}
\end{equation*}
for certain character sum $\mathfrak{C}(m,n_1,n_2)$ in \eqref{C(m,n_1,n_2)} and integral transform $\mathfrak{I}(m,n_1,n_2)$ in \eqref{eqn:I(n)definition}.
In total, we save $\left(\frac{p^{\lambda}K}{Q}\right)^{1/2}=\frac{p^{3\lambda/4-\kappa/2}K^{3/4}}{t^{1/2}}$
from the zero frequency and save
$$
\left(\frac{p^{\kappa}t^2p^{\kappa-\lambda}}{QK^2}\right)^{1/2}
\times\left(\frac{p^{\kappa+\lambda}}{p^{\frac{5\kappa}{2}-(\kappa-\lambda)}}\right)^{1/2}\times K^{1/4}=p^{\frac{\kappa-\lambda}{4}}\frac{t^{1/2}}{K^{1/2}}
$$
from the non-zero frequencies.
So the optimal choice for $\lambda$ and $K$
are given by $\lambda=3\kappa/4$ and $K=t^{4/5}$. In total, we have saved
$
N\times p^{\kappa/16}t^{1/10}.
$
It follows that
\begin{equation*}
S(N)\ll N^{1/2}p^{\frac{3}{4}}(p^{\kappa})^{1-\frac{1}{16}+\varepsilon}(|t|+1)^{1-\frac{1}{10}+\varepsilon}.
\end{equation*}

\section{Preliminaries}\label{preliminaries}
\subsection{The delta method}
Define $\delta: \mathbb{Z}\rightarrow \{0,1\}$ with
$\delta(0)=1$ and $\delta(n)=0$ for $n\neq 0$.
We will use a version of the delta method of
Duke, Friedlander and Iwaniec (see \cite[Chapter 20]{IwaniecKowalski})
which states that for any $n\in \mathbb{Z}$ and $Q\in \mathbb{R}^+$, we have
\bea\label{DFI's}
\delta(n)=\frac{1}{Q}\sum_{q\leq Q} \;\frac{1}{q}\;
\sideset{}{^\star}\sum_{a\bmod{q}}e\left(\frac{na}{q}\right)
\int_\mathbb{R}g(q,u) e\left(\frac{nu}{qQ}\right)\mathrm{d}u
\eea
where the $\star$ on the sum indicates
that the sum over $a$ is restricted to $(a,q)=1$.
The function $g$ has the following properties (see (20.158) and (20.159)
of \cite{IwaniecKowalski} and \cite[Lemma 15]{Huang})
\bea\label{g-h}
g(q,u)\ll |u|^{-A}, \;\;\;\;\;\; g(q,u) =1+h(q,u)\;\;\text{with}\;\;h(q,u)=
O\left(\frac{Q}{q}\left(\frac{q}{Q}+|u|\right)^A\right)
\eea
for any $A>1$ and
\bea\label{rapid decay g}
u^j\frac{\partial^j}{\partial u^j}g(q,u)\ll (\log Q)\min\left\{\frac{Q}{q},\frac{1}{|u|}\right\}, \qquad j\geq 1.
\eea
In particular the first property in \eqref{g-h} implies that
the effective range of the integration in
\eqref{DFI's} is $[-Q^\varepsilon, Q^\varepsilon]$.

Sun \cite{Sun} shows a new variant of delta method, which can reduce $Q$ to a smaller value by an idea of Munshi \cite{Munshi2015circleII}.

\begin{lemma}\cite[Lemma 1]{Sun}\label{sun delta}
 Suppose $\lambda \geq 1$. Let $U_0 \in \mathcal{C}_c^{\infty}(-4,4)$ be a smooth positive function with $U_0(x)=1$ if $x \in [-2, 2]$ and satisfying $U_0^{(j)}(x)\ll_j 1$. Then
\begin{equation*}\label{new variant of Delta method}
\begin{aligned}
\delta(n)= &\sum_{k=0}^{\lambda}\frac{1}{Q}\sum_{\substack{q\leq Q\\ (q,p)=1}} \frac{1}{qp^{\lambda}} \sideset{}{^\star}\sum_{a\bmod{qp^{\lambda-k}}}e\left(\frac{na}{qp^{\lambda-k}}\right)
\int_\mathbb{R}U_0\left(\frac{u}{N^{\varepsilon}}\right)g(q,u) e\left(\frac{nu}{Qqp^{\lambda}}\right)\mathrm{d}u\\
& + \sum_{r=1}^{[\log Q/ \log p]}\frac{1}{Q}\sum_{\substack{q\leq Q/p^r\\ (q,p)=1}} \frac{1}{qp^{\lambda+r}} \sideset{}{^\star}\sum_{a\bmod{qp^{\lambda+r}}}e\left(\frac{na}{qp^{\lambda+r}}\right)\\
&\int_\mathbb{R}U_0\left(\frac{u}{N^{\varepsilon}}\right)g(p^r q,u) e\left(\frac{nu}{Qqp^{\lambda+r}}\right)\mathrm{d}u + O_A(N^{-A})\\
\end{aligned}
\end{equation*}
for any $A > 0$.
\end{lemma}
\begin{proof}
See \cite[Lemma 1]{Sun}\label{sun delta}.
\end{proof}

First we recall some basic results on automorphic forms for $\mathrm{GL}_2$ (see e.g. \cite{Goldfeld}).

\subsection{Holomorphic cusp forms for $\mathrm{GL}_2$}

Let $f$ be a holomorphic cusp form of weight $k_f$ for $\rm SL_2(\mathbb{Z})$
with Fourier expansion
\begin{equation*}
 f(z)=\sum_{n=1}^{\infty}\lambda_f(n)n^{(k_f-1)/2}e(nz)
\end{equation*}
for $\mbox{Im}\,z>0$, normalized such that $\lambda_f(1)=1$.
By the Ramanujan--Petersson conjecture proved by Deligne \cite{Deligne},
we have
$
\lambda_f(n)\ll \tau(n)\ll n^{\varepsilon}
$
with $\tau(n)$ being the divisor function.

For $\varphi(x)\in \mathcal{C}_c(0,\infty)$, we set
\begin{equation}\label{intgeral transform-1}
\Phi(x) =2\pi i^{k_f} \int_0^{\infty} \varphi(y) J_{k_f-1}(4\pi\sqrt{xy})\mathrm{d}y,
\end{equation}
where $J_{k_f-1}$ is the usual $J$-Bessel function of order $k_f-1$.
We have the following Voronoi summation formula (see \cite[Theorem A.4]{KowalskiMichelVanderKam}).

\begin{lemma}\label{voronoiGL2-holomorphic}
Let $q\in \mathbb{N}$ and $a\in \mathbb{Z}$ be such
that $(a,q)=1$. For $N>0$, we have
\begin{equation*}\label{voronoi for holomorphic}
\sum_{n=1}^{\infty}\lambda_f(n)e\left(\frac{an}{q}\right)\varphi\left(\frac{n}{N}\right)
=\frac{N}{q} \sum_{n=1}^{\infty}\lambda_f(n)
e\left(-\frac{\overline{a}n}{q}\right)\Phi\left(\frac{nN}{q^2}\right),
\end{equation*}
where $\overline{a}$ denotes
the multiplicative inverse of $a$ modulo $q$.
\end{lemma}

The function $\Phi(x)$ has the following asymptotic expansion
when $x\gg 1$ (see \cite[Lemma 3.2]{LinSun}).

\begin{lemma}\label{voronoiGL2-holomorphic-asymptotic}
For any fixed integer $J\geq 1$ and $x\gg 1$, we have
\begin{equation*}
 \Phi(x)=x^{-1/4} \int_0^\infty \varphi(y)y^{-1/4}
\sum_{j=0}^{J}
\frac{c_{j} e(2 \sqrt{xy})+d_{j} e(-2 \sqrt{xy})}
{(xy)^{j/2}}\mathrm{d}y
+O_{k_f,J}\left(x^{-J/2-3/4}\right),
\end{equation*}
where $c_{j}$ and $d_{j}$ are constants depending on $k_f$.
\end{lemma}

\subsection{Maass cusp forms for $\mathrm{GL}_2$}

Let $f$ be a Hecke--Maass cusp form for $\rm SL_2(\mathbb{Z})$
with Laplace eigenvalue $1/4+t_f^2$. Then $f$ has a Fourier expansion
$$
f(z)=\sqrt{y}\sum_{n\neq 0}\lambda_f(n)K_{i t_f}(2\pi |n|y)e(nx),
$$
where $K_{i t_f}$ is the modified Bessel function of the third kind.
The Fourier coefficients satisfy
\begin{equation}\label{individual bound}
\lambda_f(n)\ll n^{\theta+\varepsilon},
\end{equation}
where, here and throughout the paper, $\theta$ denotes the exponent towards the Ramanujan
conjecture for $\rm GL_2$ Maass forms. The Ramanujan conjecture states that $\theta=0$ and
the current record due to Kim and Sarnak \cite{KimSarnak} is $\theta=7/64$.
It is well-known that
\begin{equation}\label{GL2 Rankin Selberg}
 \sum_{n\leq N}|\lambda_f(n)|^2 \ll_{f} N
\end{equation}
by the Rankin--Selberg theory.

For $\varphi(x)\in \mathcal{C}_c^{\infty}(0,\infty)$, we define the integral transforms
\begin{equation}\label{intgeral transform-2}
\begin{split}
\Phi^+(x) =& \frac{-\pi}{\sin(\pi i t_f)} \int_0^\infty \varphi(y)\left(J_{2i t_f}(4\pi\sqrt{xy})
- J_{-2i t_f}(4\pi\sqrt{xy})\right) \mathrm{d}y,\\
\Phi^-(x) =& 4\varepsilon_f\cosh(\pi t_f)\int_0^\infty \varphi(y)K_{2i t_f}(4\pi\sqrt{xy}) \mathrm{d}y,
\end{split}
\end{equation}
where $\varepsilon_f$ is an eigenvalue under the reflection operator.
We have the following Voronoi summation formula (see \cite[Theorem A.4]{KowalskiMichelVanderKam}).

\begin{lemma}\label{voronoiGL2-Maass}
Let $q\in \mathbb{N}$ and $a\in \mathbb{Z}$ be such
that $(a,q)=1$. For $N>0$, we have
\begin{equation*}
\label{voronoi for Maass form}
\sum_{n=1}^{\infty}\lambda_f(n)e\left(\frac{an}{q}\right)\varphi\left(\frac{n}{N}\right)
= \frac{N}{q} \sum_{\pm}\sum_{n=1}^{\infty}\lambda_f(n)
e\left(\mp\frac{\overline{a}n}{q}\right)\Phi^{\pm}\left(\frac{nN}{q^2}\right),
\end{equation*}
where $\overline{a}$ denotes
the multiplicative inverse of $a$ modulo $q$.
\end{lemma}

For $x\gg 1$, we have (see (3.8) in \cite{LinSun})
\begin{equation}\label{The $-$ case}
\Phi^-(x)\ll_{t_f,A}x^{-A}.
\end{equation}
For $\Phi^+(x)$ and $x\gg 1$, we have a similar asymptotic formula as for
$\Phi(x)$ in the holomorphic case (see \cite[Lemma 3.4]{LinSun}).
\begin{lemma}\label{lemma:voronoiGL2-Maass-asymptotic}
For any fixed integer $J\geq 1$ and $x\gg 1$, we have
\begin{equation*}
\Phi^{+}(x)=x^{-1/4} \int_0^\infty \varphi(y)y^{-1/4}
\sum_{j=0}^{J}
\frac{c_{j} e(2 \sqrt{xy})+d_{j} e(-2 \sqrt{xy})}
{(xy)^{j/2}}\mathrm{d}y
+O_{t_f,J}\left(x^{-J/2-3/4}\right),
\end{equation*}
where $c_{j}$ and $d_{j}$ are some constants depending on $t_f$.
\end{lemma}

\begin{remark}\label{decay-of-largeX}
For $x\gg N^{\varepsilon}$, we can choose $J$ sufficiently large so that
the contribution from the $O$-terms in Lemmas \ref{voronoiGL2-holomorphic-asymptotic} and
\ref{lemma:voronoiGL2-Maass-asymptotic}
is negligible. For the main terms
we only need to analyze the leading term $j=0$, as the analysis of the remaining
lower order terms is the same and their contribution is smaller
compared to that of the leading term.
\end{remark}

Notice that Lemma \ref{lemma:voronoiGL2-Maass-asymptotic} are only valid for $x \gg N^{\varepsilon}$. So we also need the facts which state that, for $y>0, k \geq 0$ and $\operatorname{Re} \nu =0$, one has (see \cite[Lemma C.2]{KowalskiMichelVanderKam})
\begin{equation}\label{J K Bessel function derivative}
\begin{aligned}
& y^k J_\nu^{(k)}(y) \ll_{k, \nu} \frac{1}{(1+y)^{1 / 2}}, \\
& y^k K_\nu^{(k)}(y) \ll_{k, \nu} \frac{e^{-y}(1+|\log y|)}{(1+y)^{1 / 2}} .
\end{aligned}
\end{equation}

\subsection{Estimates for exponential integrals}

Let
\begin{equation*}
 I = \int_{\mathbb{R}} w(y) e^{i h(y)} dy.
\end{equation*}
Firstly, we have the following estimates for exponential integrals
(see \cite[Lemma 8.1]{BlomerKhanYoung}).
	
	\begin{lemma}\label{lem: upper bound}
		Let $w(x)$ be a smooth function    supported on $[ a, b]$ and
        $h(x)$ be a real smooth function on  $[a, b]$. Suppose that there
		are   parameters $Q, U,   Y, Z,  R > 0$ such that
		\begin{align*}
		h^{(i)} (x) \ll_i Y / Q^{i}, \qquad w^{(j)} (x) \ll_{j } Z / U^{j},
		\end{align*}
		for  $i \geq 2$ and $j \geq 0$, and
		\begin{align*}
		| h' (x) | \geq R.
		\end{align*}
		Then for any $A \geq 0$ we have
		\begin{align*}
		I \ll_{ A} (b - a)
Z \bigg( \frac {Y} {R^2Q^2} + \frac 1 {RQ} + \frac 1 {RU} \bigg)^A .
		\end{align*}
			\end{lemma}

Next, we need the following evaluation for exponential integrals
which are
 Lemma 8.1 and Proposition 8.2 of \cite{BlomerKhanYoung} in the language of inert functions
 (see \cite[Lemma 3.1]{KiralPetrowYoung}).

Let $\mathcal{F}$ be an index set, $Y: \mathcal{F}\rightarrow\mathbb{R}_{\geq 1}$ and under this map
$T\mapsto Y_T$
be a function of $T \in \mathcal{F}$.
A family $\{w_T\}_{T\in \mathcal{F}}$ of smooth
functions supported on a product of dyadic intervals in $\mathbb{R}_{>0}^d$
is called $Y$-inert if for each $j=(j_1,\ldots,j_d) \in \mathbb{Z}_{\geq 0}^d$
we have
\begin{eqnarray*}
C(j_1,\ldots,j_d)
= \sup_{T \in \mathcal{F} } \sup_{(y_1, \ldots, y_d) \in \mathbb{R}_{>0}^d}
Y_T^{-j_1- \cdots -j_d}\left| y_1^{j_1} \cdots y_d^{j_d}
w_T^{(j_1,\ldots,j_d)}(y_1,\ldots,y_d) \right| < \infty.
\end{eqnarray*}

\begin{lemma}
\label{lemma:exponentialintegral}
 Suppose that $w = w_T(y)$ is a family of $Y$-inert functions,
 with compact support on $[Z, 2Z]$, so that
$w^{(j)}(y) \ll (Z/Y)^{-j}$.  Also suppose that $h$ is
smooth and satisfies $h^{(j)}(y) \ll H/Z^j$ for some
$H/X^2 \geq R \geq 1$ and all $y$ in the support of $w$.
\begin{enumerate}
 \item
 If $|h'(y)| \gg H/Z$ for all $y$ in the support of $w$, then
 $I \ll_A Z R^{-A}$ for $A$ arbitrarily large.
 \item If $h''(y) \gg H/Z^2$ for all $y$ in the support of $w$,
 and there exists $y_0 \in \mathbb{R}$ such that $h'(y_0) = 0$ (note $y_0$ is
 necessarily unique), then
 \begin{equation}
  I = \frac{e^{i h(y_0)}}{\sqrt{h''(y_0)}}
 F(y_0) + O_{A}(  Z R^{-A}),
 \end{equation}
where $F(y_0)$ is an $Y$-inert function (depending on $A$)  supported
on $y_0 \asymp Z$.
\end{enumerate}
\end{lemma}
Finally, we also need the $r$th derivative test with $r\geq2$ (see \cite[Lemma 5.1.3, Lemma 5.1.4]{Huxley}).
	
\begin{lemma}\label{lem: rth derivative test, dim 1}
Let $h(x)$ be real and $r$ times
differentiable on the open interval $(a, b)$
with $ h^{(r)} (x) \gg \lambda_0>0$  on $(a, b)$. Let $w(x)$
be real on $[ a, b]$ and let $V_0$ be its total
variation on $[ a, b]$ plus the maximum modulus of $w(x)$ on $[ a, b]$.
Then
		\begin{align*}
	I\ll \frac {V_0} {{\lambda_0}^{1/r}}.
		\end{align*}
	\end{lemma}

\section{Setting}
Now we start to prove Theorem \ref{main-theorem}.
\subsection{Approximate functional equation}\label{AFE}
By the approximate functional equation \cite[Theorem 5.3]{IwaniecKowalski},
we have
\begin{equation}\label{Lfunction after AFE}
L\left(\frac{1}{2}+it,f\times g \times \chi\right) \ll (p^{\kappa}t)^{\varepsilon}\sup_{N\ll (p^{\kappa}t)^{2+\varepsilon}}\frac{S(N)}{N^{1/2}}+ (p^{\kappa}t)^{-A}
\end{equation}
for $A>0$.
Here
$$
S(N)=\sum_{n=1}^{\infty}\lambda_f(n)\lambda_g(n)n^{-it}\chi(n)W\left(\frac{n}{N}\right)
$$
with $W$ some fixed smooth functions supported in $[1, 2]$ and satisfying $W^{(j)}(x)\ll_j 1$.
Note that by the Cauchy--Schwarz inequality and the Rankin--Selberg estimate \eqref{GL2 Rankin Selberg},
we have the trivial bound $S(N) \ll_{f,g,\varepsilon} N$. Therefore our goal is to obtain any extra savings on this trivial bound at this point.

\subsection{Applying DFI's delta method}\label{subsec:DFIdelta}
We write
$$
S(N)= \sum_{n=1}^{\infty}\lambda_g(n)V\left(\frac{n}{N}\right)
\sum_{m=1}^{\infty}\lambda_f(m)m^{-it}\chi(m)W\left(\frac{m}{N}\right)\delta(m-n),
$$
where $V$ is a smooth function supported in $[1/2, 5/2]$, $V(x)=1$ for $x \in [1,2]$ and $V^{(j)}(x)\ll_j 1$ for any $j \in \mathbb{N}$.
Applying Lemma \ref{sun delta} with $Q=\sqrt{N/Kp^{\lambda}}$,
we express $S(N)$ as
\begin{align*}
&\sum_{n=1}^{\infty}\lambda_g(n)V\left(\frac{n}{N}\right)
\sum_{m=1}^{\infty}\lambda_f(m)m^{-it}\chi(m)W\left(\frac{m}{N}\right)\\
&\times \bigg(\sum_{k=0}^{\lambda}\frac{1}{Q}
\sum_{\substack{q\leq Q \\ (q,p)=1}} \frac{1}{qp^{\lambda}} \sideset{}{^\star}\sum_{a\bmod{qp^{\lambda-k}}}e\left(\frac{a(m-n)}{qp^{\lambda-k}}\right) \\ & \int_\mathbb{R}U_0\left(\frac{u}{N^{\varepsilon}}\right)g(q,u) e\left(\frac{(m-n)u}{Qqp^{\lambda}}\right)\mathrm{d}u \\
&+ \sum_{r=1}^{[\log Q/ \log p]}\frac{1}{Q}
\sum_{\substack{q\leq Q/p^{r} \\ (q,p)=1}} \frac{1}{qp^{\lambda+r}} \sideset{}{^\star}\sum_{a\bmod{qp^{\lambda+r}}}e\left(\frac{a(m-n)}{qp^{\lambda+r}}\right) \\ &\int_\mathbb{R}U_0\left(\frac{u}{N^{\varepsilon}}\right)g(p^{r} q,u) e\left(\frac{(m-n)u}{Qqp^{\lambda+r}}\right)\mathrm{d}u \bigg) + O_A(N^{-A})
\end{align*}
for any $A > 0$. Note that $\lambda$ and $K$ are parameters to be determined later with $\lambda \in \mathbb{N}~(2 \leq \lambda < \kappa)$ and $0 <K< t$.

In the following we will only treat the terms with $k = 0$ and $r = 0$, because the terms with $k \neq 0$ or $r \neq 0$ are lower order terms that can be handled by the same way.
Inserting a smooth partition of unity for the $u$-integral and a dyadic partition for the $q$-sum,
we have
\begin{equation}\label{the relation between S(N)and Sb(N)}
S(N) \ll N^{\varepsilon} \sup_{X\ll N^{\varepsilon}}\sup_{1\ll R \ll Q} |S^{\flat}(N)| + N^{-A}
\end{equation}
for any $A>0$, where
\begin{equation}\label{definition of S}
\begin{aligned}
S^{\flat}(N)= & \sum_{n=1}^{\infty}\lambda_g(n)V\left(\frac{n}{N}\right)
\sum_{m=1}^{\infty}\lambda_f(m)m^{-it}\chi(m)W\left(\frac{m}{N}\right) \\
 & \times \frac{1}{Q}\sum_{\substack{q \sim R \\ (q,p)=1}} \frac{1}{qp^{\lambda}} \sideset{}{^\star}\sum_{a\bmod{qp^{\lambda}}}e\left(\frac{a(m-n)}{qp^{\lambda}}\right) \int_\mathbb{R}U\left(\frac{\pm u}{X}\right)g(q,u) e\left(\frac{(m-n)u}{Qqp^{\lambda}}\right)\mathrm{d}u.
\end{aligned}
\end{equation}
Here $U$ is a fixed compactly supported 1-inert function with $\supp U \subset [1,2]$.

Rearranging the order of the sums and integrals in \eqref{definition of S} we get
\begin{equation}\label{definition of m and n sums}
\begin{aligned}
S^{\flat}(N)= &\frac{1}{Q}\sum_{\substack{q \sim R \\ (q,p)=1}} \frac{1}{qp^{\lambda}} \sideset{}{^\star}\sum_{a\bmod{qp^{\lambda}}}
\int_\mathbb{R}U\left(\frac{\pm u}{X}\right)g(q,u)
\sum_{n=1}^{\infty}\lambda_{g}(n)e\left(-\frac{na}{qp^{\lambda}}\right)
e\left(-\frac{nu}{Qqp^{\lambda}}\right)V\left(\frac{n}{N}\right)\\
& \times \sum_{m=1}^{\infty}\lambda_{f}(m)m^{-it}\chi(m)e\left(\frac{ma}{qp^{\lambda}}\right)
e\left(\frac{mu}{Qqp^{\lambda}}\right)W\left(\frac{m}{N}\right)\mathrm{d}u.
\end{aligned}
\end{equation}
Without loss of generality,
we only consider the contribution from $u > 0$ and $t>0$ (the contribution from $u<0$ or $t<0$ can be
estimated similarly). By abuse of notation, we still write the contribution from $u > 0$ and $t>0$
as $S^{\flat}(N)$.

\section{Voronoi summation formulas}\label{Voronoi}
In this section, we transform the $m$ and $n$ sums in \eqref{definition of m and n sums} by Voronoi summation formulas.

We first apply the $\rm GL_2$ Voronoi summation formula to the $n$-sum
$$
S_1:= \sum_{n=1}^{\infty}\lambda_{g}(n)e\left(-\frac{na}{qp^{\lambda}}\right)
e\left(-\frac{nu}{Qqp^{\lambda}}\right)V\left(\frac{n}{N}\right).
$$
Applying Lemma \ref{voronoiGL2-Maass} with $\varphi_1(x)= e\left(-\frac{Nux}{Qqp^{\lambda}}\right)V(x)$, we obtain that
\begin{equation}\label{n sum after GL2 Voronoi}
S_1= \frac{N}{qp^{\lambda}}\sum_{\pm}\sum_{n=1}^{\infty}\lambda_g(n)
e\left(\pm\frac{\overline{a}n}{qp^{\lambda}}\right)
\Phi_1^{\pm}\left(\frac{nN}{q^2p^{2\lambda}}\right)
\end{equation}
where if $g$ is holomorphic, $\Phi_{1}^+(y)=\Phi_{1}(y)$ with $\Phi_{1}(y)$
given by \eqref{intgeral transform-1} and $\Phi_{1}^-(y)=0$,
while for $g$ a Hecke--Maass cusp form, $\Phi_{1}^{\pm}(y)$ are given by \eqref{intgeral transform-2}.

\begin{lemma}\label{Phi 1 analysis}
  \begin{enumerate}
  \item For $y= \frac{nN}{q^2 p^{2\lambda}}\gg N^{\varepsilon}$, we have $\Phi_{1}^-(y)\ll N^{-A}$.
  For $y= \frac{nN}{q^2 p^{2\lambda}}\gg N^{\varepsilon}$ and $n\asymp \frac{Nu^2}{Q^2}\asymp \frac{NX^2}{Q^2}$, we have
      $$\Phi_1^{+}(y)= y^{-\frac{1}{2}}V_1\left(y\frac{Q^2q^2p^{2\lambda}}{N^2u^2}\right)e\left(y\frac{Qqp^{\lambda}}{Nu}\right)+ O(N^{-A}).$$

  \item For $y= \frac{nN}{q^2 p^{2\lambda}}\ll N^{\varepsilon}$ and $\frac{NX}{p^{\lambda}RQ}\gg N^{\varepsilon}$, we have $\Phi_1^{\pm}(y)\ll N^{-A}.$
  \item For $y= \frac{nN}{q^2 p^{2\lambda}}\ll N^{\varepsilon}$ and $\frac{NX}{p^{\lambda}RQ}\ll N^{\varepsilon}$, we have  $\Phi_1^{\pm}(y)\ll N^{\varepsilon}.$
  \end{enumerate}
\end{lemma}

\begin{proof}
 For $y= \frac{nN}{q^2 p^{2\lambda}}\gg N^{\varepsilon}$, we know that
$\Phi_{1}^-(y)\ll N^{-A}$ by \eqref{The $-$ case}, and
\begin{equation}\label{voronoiGL2-Maass-asymptotic1}
\Phi_1^{+}(y)=y^{-\frac{1}{4}}\int_{0}^{\infty}V(\xi)\xi^{-\frac{1}{4}}e\left(-\frac{Nu}{Qqp^{\lambda}}\xi+2\sigma_1 \sqrt{y\xi}\right)\mathrm{d}\xi + O(N^{-A})
\end{equation}
with $\sigma_1 = \pm 1$.
For $\Phi_1^{+}\left(y\right)$, it is negligibly small by repeated integration by parts unless $\sqrt{y}\asymp \frac{Nu}{Qqp^{\lambda}}$ and $\sgn \sigma_1 = +$, in which case we apply
stationary phase method to the $\xi$-integral in \eqref{voronoiGL2-Maass-asymptotic1}.
Let
$$
h_1(\xi) = -\frac{Nu}{Qqp^{\lambda}}\xi + 2\sqrt{y\xi}.
$$
Then we have
$$
  h_1'(\xi) = -\frac{Nu}{Qqp^{\lambda}} +  \sqrt{y} \xi^{-1/2}, \quad
  h_1''(\xi) =  - \frac{1}{2} \sqrt{y} \xi^{-3/2} \asymp \sqrt{y}, \quad
  h_1^{(j)}(\xi) \ll_j \sqrt{y} , \ j\geq3.
$$
Note that the stationary point is $\xi_{01}= y\frac{Q^2q^2p^{2\lambda}}{N^2u^2}$ and we have
$h_1(\xi_{01})= y\frac{Qqp^{\lambda}}{Nu}$.
Let $w_1(\xi) = V\left(\xi\right) \xi^{-1/4}$. Then we have $w_1^{(j)}(\xi) \ll_j 1$ and $\supp w_1\subset [1/4,4]$.
By Lemma \ref{lemma:exponentialintegral}(2), we obtain that
\begin{equation}\label{Phi 1 after stationary phase}
\Phi_1^{+}(y)= y^{-\frac{1}{2}}V_1\left(y\frac{Q^2q^2p^{2\lambda}}{N^2u^2}\right)e\left(y\frac{Qqp^{\lambda}}{Nu}\right)+ O(N^{-A}),
\end{equation}
where  $V_1$ is an $1$-inert function (depending on $A$)
supported on $\xi_{01} \asymp 1$. Hence, case (1) is proved.

For $y=\frac{nN}{q^2p^{2\lambda}}\ll N^{\varepsilon}$, by \eqref{intgeral transform-2} we may regard $\Phi_1^{\pm}(y)$ as
\begin{equation}\label{Phi 1 y small}
\Phi_1^{\pm}(y)=\int_{0}^{\infty}W(\xi)
e\left(- \frac{Nu\xi}{Qqp^{\lambda}}\right)J_g^{\pm}(y\xi)\mathrm{d}\xi,
\end{equation}
where
$$
J_g^{+}(y) = \frac{-\pi}{\sin(\pi i t_g)} \left(J_{2it_g}(4\pi\sqrt{y})- J_{-2it_g}(4\pi\sqrt{y})\right),
$$
and
$$
J_g^{-}(y) = 4\varepsilon_g\cosh(\pi t_g)K_{2it_g}(4\pi\sqrt{y}).
$$
Then, by partial integration and \eqref{J K Bessel function derivative}, $\Phi_1^{\pm}(y)$ is negligible unless $\frac{Nu}{Qqp^{\lambda}} \ll N^{\varepsilon}$, in which case $\Phi_1^{\pm}(y)\ll N^{\varepsilon}.$
Therefore, case (2) and (3) are proved.
\end{proof}

According to Lemma \ref{Phi 1 analysis}, we first consider the estimations for the case
$\frac{nN}{q^2 p^{2\lambda}}\asymp \big(\frac{NX}{RQp^{\lambda}}\big)^2 \gg N^{\varepsilon}$.
For the complementary range $\frac{nN}{q^2 p^{2\lambda}}\ll N^{\varepsilon}$ and $\frac{NX}{p^{\lambda}RQ}\ll N^{\varepsilon}$, we will analyse the contribution of the case in \S \ref{The complementary cases}.

Plugging \eqref{Phi 1 after stationary phase} into \eqref{n sum after GL2 Voronoi}, we find that
the $n$-sum is asymptotically equal to
$$
 N^{\frac{1}{2}}\sum_{n}\frac{\lambda_g(n)}{n^{1/2}}
e\left(\frac{\overline{a}n}{qp^{\lambda}}\right)
V_1\left(\frac{n Q^2}{Nu^2}\right)e\left(\frac{nQ}{qp^{\lambda}u}\right),
$$
and we have $n\asymp \frac{Nu^2}{Q^2}\asymp \frac{NX^2}{Q^2}$.
Hence
\begin{equation*}
\begin{aligned}
S^{\flat}(N)=
& \frac{N^{\frac{1}{2}}}{Q}\sum_{\substack{q\sim R \\ (q,p)=1}} \frac{1}{qp^{\lambda}} \sideset{}{^\star}\sum_{a\bmod{qp^{\lambda}}}
  \sum_{n\asymp \frac{NX^2}{Q^2}}\frac{\lambda_g(n)}{n^{1/2}}
e\left(\frac{\overline{a}n}{qp^{\lambda}}\right)
  \sum_{m}\lambda_f(m)\chi(m)m^{-it}
 e\left(\frac{am}{qp^{\lambda}}\right)V\left(\frac{m}{N}\right)\\
& \cdot \int_{\mathbb{R}}U\left(\frac{u}{X}\right)g(q,u) V_1\left(\frac{n Q^2}{Nu^2}\right)e\left(\frac{nQ}{qp^{\lambda}u}+\frac{mu}{Qqp^{\lambda}}\right)\mathrm{d}u + O(N^{-A}).
\end{aligned}
\end{equation*}
Making a change of variable $u= X\xi$, we have
\begin{equation}\label{I defination}
I(X)= X \int_{\mathbb{R}}U(\xi)g(q,X\xi) V_1\left(\frac{n Q^2}{NX^2\xi^2}\right)e\left(\frac{nQ}{p^{\lambda}qX\xi}+\frac{mX\xi}{Qqp^{\lambda}}\right)\mathrm{d}\xi.
\end{equation}

Next, we apply stationary phase method to \eqref{I defination}.
Let
$$
h_2(\xi) = \frac{nQ}{p^{\lambda}qX\xi}+\frac{mX\xi}{Qqp^{\lambda}}.
$$
Then we have
$$
  h_2'(\xi) = -\frac{nQ}{p^{\lambda}qX\xi^2} + \frac{mX}{Qqp^{\lambda}}, \quad
  h_2''(\xi) =  \frac{2nQ}{p^{\lambda}qX\xi^3}  \asymp \frac{NX}{p^{\lambda}RQ}.
$$
Let $h_2'(\xi_{02}) = 0$. Then we have the stationary point $\xi_{02}= \frac{n^{1/2}Q}{m^{1/2}X}$ and
$h_2(\xi_{02})= 2\frac{m^{1/2}n^{1/2}}{qp^{\lambda}}$.
By Lemma \ref{lemma:exponentialintegral}(2), we obtain that
\begin{equation}\label{Phi after stationary phase}
I(X)=X \left(\frac{nQ}{p^{\lambda}qX}\right)^{-1/2}V_2\left(\frac{nQ^2}{mX^2}\right)
e\left(2\frac{m^{1/2}n^{1/2}}{qp^{\lambda}}\right)
+ O(N^{-A}),
\end{equation}
where $V_2$ is an 1-inert function (depending on $A$) supported on $\xi_{02} \asymp 1$.
Hence
\begin{equation}\label{Sb(N) before S2}
\begin{aligned}
S^{\flat}(N)=
& \frac{X^{\frac{3}{2}}N^{\frac{1}{2}}}{Q^{\frac{3}{2}}}\sum_{\substack{q\sim R \\ (q,p)=1}} \frac{1}{(qp^{\lambda})^{1/2}} \sideset{}{^\star}\sum_{a\bmod{qp^{\lambda}}}
  \sum_{n\asymp \frac{NX^2}{Q^2}}\frac{\lambda_g(n)}{n}
e\left(\frac{\overline{a}n}{qp^{\lambda}}\right) \\
& \cdot \sum_{m}\lambda_f(m)\chi(m)m^{-it}
 e\left(\frac{am}{qp^{\lambda}}\right)V\left(\frac{m}{N}\right)
 V_2\left(\frac{n Q^2}{mX^2}\right)
e\left(2\frac{m^{1/2}n^{1/2}}{qp^{\lambda}}\right)+ O(N^{-A}).
\end{aligned}
\end{equation}
Let
$$
S_2:= \sum_{m}\lambda_f(m)\chi(m)m^{-it}
 e\left(\frac{am}{qp^{\lambda}}\right)
 e\left(2\frac{m^{1/2}n^{1/2}}{qp^{\lambda}}\right)
 V\left(\frac{m}{N}\right)V_2\left(\frac{n Q^2}{mX^2}\right).
$$
By the Fourier expansion of $\chi$ in terms of additive characters, we have
$$
S_2= \frac{1}{\tau(\overline{\chi})}\sum_{c \mod p^{\kappa}}\overline{\chi}(c)\sum_{m}\lambda_f(m)
e\left(\frac{(ap^{\kappa-\lambda}+cq)m}{qp^{\kappa}}\right)m^{-it}
e\left( 2\frac{m^{1/2}n^{1/2}}{qp^{\lambda}}\right)
V\left(\frac{m}{N}\right)V_2\left(\frac{n Q^2}{mX^2}\right).
$$
Note that $(ap^{\kappa-\lambda}+cq, q)=1$ and $(ap^{\kappa-\lambda}+cq, p)=(cq,p)=1$. Thus $(ap^{\kappa-\lambda}+cq, qp^{\kappa})=1$. Denote $a^{\ast}= ap^{\kappa-\lambda}+cq$.

Next, we apply the $\rm GL_2$ Voronoi summation formula in Lemma \ref{voronoiGL2-Maass} with
$$
\varphi_2(x)=x^{-it}
e\left(2\frac{N^{1/2}n^{1/2}}{qp^{\lambda}}x^{1/2}\right)
V\left(x\right)V_2\left(\frac{n Q^2}{NX^2x}\right)
$$
and obtain that
$$
S_2= \frac{1}{\tau(\overline{\chi})}\frac{N^{1-it}}{qp^{\kappa}}\sum_{c \mod p^{\kappa}}\overline{\chi}(c)\sum_{m}\lambda_f(m)
e\left(\frac{-\overline{a^{\ast}}m}{qp^{\kappa}}\right)
\Phi_2^{\pm}\left(\frac{mN}{q^2p^{2\kappa}}\right),
$$
where $\Phi_{2}^{\pm}(y)$ are given by \eqref{intgeral transform-2}.

\begin{lemma}\label{Phi 2 m small}
  If $y= \frac{mN}{q^2p^{2\kappa}}\ll N^{\varepsilon}$, then $\Phi_2^{\pm}(y)$ is negligible unless $t\asymp \frac{N^{1/2}n^{1/2}}{qp^{\lambda}}\asymp \frac{NX}{p^{\lambda}RQ}$.
\end{lemma}
\begin{proof}
For $\frac{mN}{q^2p^{2\kappa}}\ll N^{\varepsilon}$, by \eqref{intgeral transform-2} we may regard $\Phi_2^{\pm}(y)$ as
\begin{equation}\label{psi2 case b}
\Phi_2^{\pm}(y)=\int_{0}^{\infty}V(\xi)V_2\left(\frac{n Q^2}{NX^2\xi}\right)
e\left(-\frac{t\log \xi}{2\pi}+ 2\frac{N^{1/2}n^{1/2}}{qp^{\lambda}}\xi^{1/2}\right)J_f^{\pm}(y\xi)\mathrm{d}\xi,
\end{equation}
where
$$
J_f^{+}(y) = \frac{-\pi}{\sin(\pi i t_f)} \left(J_{2it_f}(4\pi\sqrt{y})- J_{-2it_f}(4\pi\sqrt{y})\right),
$$
and
$$
J_f^{-}(y) = 4\varepsilon_f\cosh(\pi t_f)K_{2it_f}(4\pi\sqrt{y}).
$$
Then, by partial integration together with \eqref{J K Bessel function derivative}, $\Phi_2^{\pm}(y)$ is negligible unless $t\asymp \frac{N^{1/2}n^{1/2}}{qp^{\lambda}}\asymp \frac{NX}{p^{\lambda}RQ}$.
\end{proof}

We will analyse the contribution of the case $\frac{mN}{q^2p^{2\kappa}}\ll N^{\varepsilon}$ to $S(N)$ in \S \ref{The complementary cases}.
In this section we consider the estimations for the case $\frac{mN}{q^2 p^{2\kappa}}\gg N^{\varepsilon}$.

For $y= \frac{mN}{q^2 p^{2\kappa}}\gg N^{\varepsilon}$, in view of \eqref{The $-$ case}, $\Phi_2^{-}(y)$ is negligible. By Lemma \ref{lemma:voronoiGL2-Maass-asymptotic}, we have
\begin{equation}\label{M voronoi asymptotic}
\Phi_2^{+}(y)=y^{-\frac{1}{4}}\int_{0}^{\infty}V(\xi)V_2\left(\frac{n Q^2}{NX^2\xi}\right)\xi^{-\frac{1}{4}}
e\left(-\frac{t\log \xi}{2\pi}+ 2\frac{N^{1/2}n^{1/2}}{qp^{\lambda}}\xi^{1/2}
+ 2\sigma_2 \sqrt{y\xi}\right)\mathrm{d}\xi + O(N^{-A})
\end{equation}
with $\sigma_2=\pm 1$.
Making a change of variable $\xi^{1/2}\rightarrow \xi$ and taking $y= \frac{mN}{q^2 p^{2\kappa}}$, we obtain
$$
\begin{aligned}
\Phi_2^{+}\left(\frac{mN}{q^2 p^{2\kappa}}\right)=2\left(\frac{mN}{q^2 p^{2\kappa}}\right)^{-\frac{1}{4}}\int_{0}^{\infty}
&V_3(\xi) e\left(-\frac{t\log \xi}{\pi}+ 2\frac{N^{1/2}n^{1/2}}{qp^{\lambda}}\xi
+ 2\sigma_2 \frac{\sqrt{mN}}{q p^{\kappa}}\xi\right)\mathrm{d}\xi + O(N^{-A}),
\end{aligned}
$$
where $V_3(\xi)= V(\xi^2)V_2\left(\frac{n Q^2}{NX^2\xi^2}\right)\xi^{\frac{1}{2}}$ is an 1-inert function.

By repeated integration by parts, one can truncate the $m$-sum at
$m\ll \max\{\frac{t^2R^2p^{2\kappa}}{N}, \frac{NX^2p^{2\kappa-2\lambda}}{Q^2}\}.$
Hence, it is sufficient to estimate
\begin{equation}\label{Sb(N) before Cauchy 1}
\begin{aligned}
& \frac{N^{5/4-it}X^{3/2}}{Q^{3/2}\tau(\overline{\chi})p^{(\kappa+\lambda)/2}}
\sum_{\substack{q\sim R \\ (q,p)=1}} \frac{1}{q}  \sum_{n\asymp\frac{NX^2}{Q^2}}\frac{\lambda_g(n)}{n}
\sum_{m\ll \max\{\frac{t^2R^2p^{2\kappa}}{N}, \frac{NX^2p^{2\kappa-2\lambda}}{Q^2}\}}\frac{\lambda_f(m)}{m^{1/4}} \sideset{}{^\star}\sum_{a\bmod{qp^{\lambda}}}\left(\frac{\overline{a}n}{qp^{\lambda}}\right)\\
&\cdot\sum_{c \mod p^{\kappa}}\overline{\chi}(c)e\left(\frac{-\overline{a^{\ast}}m}{qp^{\kappa}}\right)
\int_{0}^{\infty}V_3(\xi)e\left(-\frac{t\log \xi}{\pi}+ 2\frac{N^{1/2}n^{1/2}}{qp^{\lambda}}\xi
+2\sigma_2\frac{N^{1/2}m^{1/2}}{qp^{\kappa}}\xi\right)\mathrm{d}\xi.
\end{aligned}
\end{equation}
Note that
$$
e\left(\frac{-\overline{a^{\ast}}m}{qp^{\kappa}}\right)=
e\left(\frac{-(\overline{ap^{\kappa-\lambda}+cq})m}{qp^{\kappa}}\right)
=e\left(\frac{-(\overline{ap^{\kappa-\lambda}+cq})m\bar{q}}{p^{\kappa}}\right)
e\left(\frac{-\overline{ap^{2\kappa-\lambda}}m}{q}\right).
$$
The character sum in \eqref{Sb(N) before Cauchy 1} is
\begin{equation}\label{character sum}
\begin{aligned}
&\sideset{}{^\star}\sum_{a\bmod{qp^{\lambda}}}\left(\frac{\overline{a}n}{qp^{\lambda}}\right)
\sum_{c \mod p^{\kappa}}\overline{\chi}(c)e\left(\frac{-\overline{a^{\ast}}m}{qp^{\kappa}}\right)\\
= &\sum_{c \mod p^{\kappa}}\overline{\chi}(c)\sideset{}{^\star}\sum_{b\bmod{p^{\lambda}}}
e\left(\frac{-(\overline{ap^{\kappa-\lambda}+cq})m\bar{q}}{p^{\kappa}}+ \frac{n\bar{q}\bar{b}}{p^{\lambda}}\right)
\sum_{\substack{d|q \\ n\equiv mp^{\lambda}\overline{p^{2\kappa-\lambda}}\bmod d}}d\mu\left(\frac{q}{d}\right).
\end{aligned}
\end{equation}
Plugging \eqref{character sum} into \eqref{Sb(N) before Cauchy 1}, we are led to estimate
\begin{equation}\label{Sb(N) before Cauchy 2}
\begin{aligned}
\frac{N^{5/4-it}X^{3/2}}{Q^{3/2}\tau(\overline{\chi})p^{(\kappa+\lambda)/2}}
\sum_{\substack{q\sim R \\ (q,p)=1}} \frac{1}{q}
&\sum_{d|q}d\mu\left(\frac{q}{d}\right)
\sum_{n\asymp\frac{NX^2}{Q^2}}\frac{\lambda_g(n)}{n} \\
&\times \sum_{\substack{m\ll \max\{\frac{t^2R^2p^{2\kappa}}{N}, \frac{NX^2p^{2\kappa-2\lambda}}{Q^2}\} \\ n\equiv mp^{\lambda}\overline{p^{2\kappa-\lambda}}\bmod d}}\frac{\lambda_f(m)}{m^{1/4}} \mathfrak{C}^{\ast}(m,n,q)I(m,n,q)
\end{aligned}
\end{equation}
with
$$
\mathfrak{C}^{\ast}(m,n,q):= \sum_{c \mod p^{\kappa}}\overline{\chi}(c)
\sideset{}{^\star}\sum_{b\bmod{p^{\lambda}}}
e\left(\frac{-(\overline{ap^{\kappa-\lambda}+cq})m\bar{q}}{p^{\kappa}}+ \frac{n\bar{q}\bar{b}}{p^{\lambda}}\right)
$$
and
$$
I(m,n,q):= \int_{0}^{\infty}V_3(\xi)e\left(-\frac{t\log \xi}{\pi}+ 2\frac{N^{1/2}n^{1/2}}{qp^{\lambda}}\xi
+2\sigma_2\frac{N^{1/2}m^{1/2}}{qp^{\kappa}}\xi\right)\mathrm{d}\xi.
$$

\section{Applying Cauchy--Schwarz inequality and Poisson summation formula}\label{CauchyPoisson}
Applying the Cauchy--Schwarz inequality to the $m$-sum in \eqref{Sb(N) before Cauchy 2} and using the
Rankin–-Selberg estimate in \eqref{GL2 Rankin Selberg}, we get
\begin{equation}\label{Sb(N) after CAUCHY}
\begin{aligned}
  S^{\flat}(N)&\ll \sup_{M_1\ll \max\{\frac{t^2R^2p^{2\kappa}}{N}, \frac{NX^2p^{2\kappa-2\lambda}}{Q^2}\}}\frac{N^{5/4}X^{3/2}}{Q^{3/2}p^{(2\kappa+\lambda)/2}}
\sum_{\substack{q\sim R \\ (q,p)=1}} \frac{1}{q} \sum_{d|q}d \\
& \quad \quad \quad \quad \cdot \sum_{m\sim M_1 }\frac{\mid\lambda_f(m)\mid}{m^{1/4}}\bigg|\sum_{\substack{n\asymp\frac{NX^2}{Q^2} \\ n\equiv mp^{\lambda}\overline{p^{2\kappa-\lambda}}\bmod d}}\frac{\lambda_g(n)}{n} \mathfrak{C}^{\ast}(m,n,q)I(m,n,q)\bigg|\\
  &\ll\sup_{M_1\ll \max\{\frac{t^2R^2p^{2\kappa}}{N}, \frac{NX^2p^{2\kappa-2\lambda}}{Q^2}\}}
  \frac{N^{5/4}X^{3/2}M_1^{1/4}}{Q^{3/2}p^{(2\kappa+\lambda)/2}}\sum_{\substack{q\sim R\\ (q,p)=1}} \frac{1}{q}\sum_{d|q}d \cdot \Sigma(d,q)^{1/2},
\end{aligned}
\end{equation}
where
$$
  \Sigma(d,q) :=
  \sum_{m}
W_1\left(\frac{m}{M_1}\right)
   \bigg| \sum_{\substack{n\asymp\frac{NX^2}{Q^2} \\ n\equiv mp^{\lambda}\overline{p^{2\kappa-\lambda}}\bmod d}}\frac{\lambda_g(n)}{n} \mathfrak{C}^{\ast}(m,n,q)I(m,n,q) \bigg|^2
$$
with $W_1$ supported on $[1, 2]$ and satisfying $W_1^{(j)}(x) \ll_j 1$.
Expanding the absolute square and switching the order of summations, we obtain
\begin{multline*}
  \Sigma(d,q) =
  \sum_{n_1\asymp \frac{NX^2}{Q^2} } \frac{\lambda_g(n_1)}{n_1}
  \sum_{\substack{n_2\asymp \frac{NX^2}{Q^2} \\ n_2\equiv n_1 \bmod d}}
   \frac{\lambda_g(n_2)}{n_2}
  \sum_{\substack{m \geq1 \\ m\equiv n_1\overline{p^{\lambda}}p^{2\kappa-\lambda} \bmod d}}
  W_1\left(\frac{m}{M_1}\right) \\ \cdot\mathfrak{C}^{\ast}(m,n_1,q)\overline{\mathfrak{C}^{\ast}(m,n_2,q)}
  I(m,n_1,q)\overline{I(m,n_2,q)} .
\end{multline*}
Note that $\lambda_g(n_1) \lambda_g(n_2) \ll |\lambda_g(n_1)|^2 + |\lambda_g(n_2)|^2$. Hence we have
\begin{equation}\label{Sigma(d,q)}
 \Sigma(d,q) \ll  \frac{Q^4}{N^2 X^4}
  \sum_{n_1\asymp \frac{NX^2}{Q^2} } |\lambda_g(n_1)|^2
  \sum_{\substack{n_2\asymp \frac{NX^2}{Q^2} \\ n_2\equiv n_1 \bmod d}}| \mathcal{K} |,
\end{equation}
where
$$
\mathcal{K}=\sum_{\substack{m \geq1 \\ m\equiv n_1\overline{p^{\lambda}}p^{2\kappa-\lambda} \bmod d}}
  W_1\left(\frac{m}{M_1}\right) \mathfrak{C}^{\ast}(m,n_1,q)\overline{\mathfrak{C}^{\ast}(m,n_2,q)}
  I(m,n_1,q)\overline{I(m,n_2,q)}.
$$
Applying the Poisson summation formula with modulo $dp^{\kappa}$, we get
\begin{equation}\label{K after Poisson}
\mathcal{K} = \frac{M_1}{dp^{\kappa}} \sum_{m\in\mathbb{Z}} \mathfrak{C}(m,n_1,n_2)\cdot \mathfrak{I}(m,n_1,n_2),
\end{equation}
where
\begin{equation}\label{C(m,n_1,n_2)}
  \mathfrak{C}(m,n_1,n_2) = \sum_{ \gamma \bmod p^{\kappa}}\mathfrak{C}^{\ast}(\gamma,n_1,q)\overline{\mathfrak{C}^{\ast}(\gamma,n_2,q)}
  e\left(\frac{m\left(n_1p^{3\kappa-\lambda}\overline{p^{\kappa+\lambda}}+\gamma d\bar{d}\right)}{dp^{\kappa}}\right)
\end{equation}
and
\begin{equation}\label{eqn:I(n)definition}
  \mathfrak{I}(m,n_1,n_2) := \int_{\mathbb{R}}   W_1\left( z\right)
  I(M_1z,n_1,q)\overline{I(M_1z,n_2,q)}e\left(-\frac{mM_1}{p^{\kappa}d} z \right)  \dd z
\end{equation}
with
\begin{equation}\label{I(M_1z,n,q)definition}
I(M_1z,n,q)= \int_{0}^{\infty}V_3(\xi)e\left(-\frac{t\log \xi}{\pi}+ 2\frac{N^{1/2}n^{1/2}}{qp^{\lambda}}\xi
+2\sigma_2\frac{N^{1/2}M_1^{1/2}z^{1/2}}{qp^{\kappa}}\xi\right)\mathrm{d}\xi.
\end{equation}
For the character sum, Sun \cite[Lemma 4]{Sun} has the following estimation:

\begin{lemma}\label{character sum estimate}
Assume $(2\kappa+1)/3\leq \lambda\leq 3\kappa/4$. The character sum vanishes unless $(n_1n_2,p)=1$ and
$p^{\kappa-\lambda}|m$.
Let $m=m'p^{\kappa-\lambda}$, $\lambda=2\alpha+\delta_1$ and
$\kappa=2\beta+\delta_2$ with $\delta_1,\delta_2=0$ or 1, $\alpha\geq 1$ and $\beta\geq 1$.

(1) If $p^{\beta-\kappa+\lambda}|m'$, then $n_1\equiv n_2 (\bmod \,p^{\beta-\kappa+\lambda})$ and
\begin{equation*}
 \mathfrak{C}(m,n_1,n_2) \ll p^{2\kappa+\lambda+\delta_1}.
\end{equation*}

(2) If $p^{\ell}\| m'$ with $\ell<\beta-\kappa+\lambda$, then $p^{\ell}\| n_1-n_2$ and
\begin{equation*}
\mathfrak{C}(m,n_1,n_2) \ll p^{5\kappa/2+2\ell +\delta_1+\delta_2/2}.
\end{equation*}

(3) For $m=0$, we have
\begin{equation*}
\mathfrak{C}(m,n_1,n_2)=p^{2\kappa}\sum_{d|(n_1-n_2,p^{\lambda})}d\mu\left(p^{\lambda}/d\right).
\end{equation*}

\end{lemma}
Next, we deal with $\mathfrak{I}(m,n_1,n_2)$.

\subsection{The large modulo case}
Assume $N^{\varepsilon}\ll\frac{NX}{p^{\lambda}RQ}\ll t^{1-\varepsilon}$. Then we have
$$
\frac{N^{1/2}n^{1/2}}{qp^{\lambda}}\asymp \frac{NX}{p^{\lambda}RQ}\ll t^{1-\varepsilon}.
$$
We first consider $I(M_1z,n,q)$.
For $I(M_1z,n,q)$, it is negligibly small by repeated integration by parts unless $t \asymp \frac{N^{1/2}M_1^{1/2}z^{1/2}}{qp^{\kappa}}$ (i.e. $M_1\asymp \frac{t^2R^2p^{2\kappa}}{N}$) and $\sgn \sigma_2 = +$, in which case we apply the stationary phase method to the $\xi$-integral.
Let $h_3(\xi)= -\frac{t\log \xi}{\pi}+ 2\frac{N^{1/2}n^{1/2}}{qp^{\lambda}}\xi
+2\frac{N^{1/2}M_1^{1/2}z^{1/2}}{qp^{\kappa}}\xi$.
Then we have
$$
  h_3'(\xi) = -\frac{t}{\pi \xi}+ 2\frac{N^{1/2}n^{1/2}}{qp^{\lambda}}
+ 2\frac{N^{1/2}M_1^{1/2}z^{1/2}}{qp^{\kappa}}, \quad
  h_3''(\xi) =  \frac{t}{\pi \xi^2}  \asymp t.
$$
The stationary point is
$$\xi_{03}= \frac{t}{2\pi (\frac{N^{1/2}M_1^{1/2}z^{1/2}}{qp^{\kappa}}+ \frac{N^{1/2}n^{1/2}}{qp^{\lambda}})},$$
and
$$
h_3(\xi_{03})= \frac{t}{\pi}\log\frac{2 e\pi(\frac{N^{1/2}M_1^{1/2}z^{1/2}}{qp^{\kappa}}+ \frac{N^{1/2}n^{1/2}}{qp^{\lambda}})}{t}.
$$
By Lemma \ref{lemma:exponentialintegral}(2), we have
\begin{equation}\label{I(M_1z,n,q) after stationary phase}
I(M_1z,n,q)=t^{-\frac{1}{2}}
V_4\left(\frac{\frac{N^{1/2}M_1^{1/2}z^{1/2}}{qp^{\kappa}}+ \frac{N^{1/2}n^{1/2}}{qp^{\lambda}}}{t/2\pi}\right)
e\left(\frac{t}{\pi}\log\frac{2 e\pi(\frac{N^{1/2}M_1^{1/2}z^{1/2}}{qp^{\kappa}}+  \frac{N^{1/2}n^{1/2}}{qp^{\lambda}})}{t}\right),
\end{equation}
where $V_4$ is an 1-inert function supported on $\xi_{03}\asymp 1$.
The oscillation in the above exponential function is of size $t$, which is quite large. Next, we will reduce the oscillation. This is one of the key observations in our paper. Let
\begin{equation}\label{h(M_1z,n,q)defination}
\begin{aligned}
  h(M_1z,n,q):=&\frac{t}{\pi}\log\frac{2 e\pi(\frac{N^{1/2}M_1^{1/2}z^{1/2}}{qp^{\kappa}}+\frac{N^{1/2}n^{1/2}}{qp^{\lambda}})}{t}\\
  =&\frac{t}{\pi}\log\frac{2e\pi}{t}+\frac{t}{\pi}\log\frac{N^{1/2}M_1^{1/2}z^{1/2}}{qp^{\kappa}}
  +\frac{t}{\pi} \log \left( 1+ \frac{p^{\kappa-\lambda}n^{1/2}}{(M_1z)^{1/2}}\right).
\end{aligned}
\end{equation}
Since
$$
  p^{2\kappa-2\lambda}n\asymp  p^{2\kappa-2\lambda}\frac{NX^2}{Q^2}  \ll p^{2\kappa-2\lambda}\frac{R^2p^{2\lambda}t^2}{N} t^{-2\varepsilon}
  \asymp  M_1 t^{-2\varepsilon},
$$
by the Taylor expansion we get
\begin{equation}\label{Taylor expansion}
  \log \left( 1+ \frac{p^{\kappa-\lambda}n^{1/2}}{(M_1z)^{1/2}}\right)
  =  \frac{p^{\kappa-\lambda}n^{1/2}}{(M_1z)^{1/2}} + \frac{p^{2\kappa-2\lambda}n}{2M_1z} + \sum_{j\geq3} c_j \left(  \frac{p^{\kappa-\lambda}n^{1/2}}{(M_1z)^{1/2}}\right)^j.
\end{equation}
By \eqref{h(M_1z,n,q)defination} and \eqref{Taylor expansion}, we have
\begin{equation}\label{h(M_1z,n,q)}
  h(M_1z,n,q)=h_1(M_1z,q) + h_2(M_1z,n,q),
\end{equation}
where
$$
  h_1(M_1z,q) =\frac{t}{\pi}\log\frac{2e\pi}{t}+\frac{t}{\pi}\log\frac{N^{1/2}M_1^{1/2}z^{1/2}}{qp^{\kappa}}
$$
and
$$
  h_2(M_1z,n,q)=
    \frac{t}{\pi}\left(\frac{p^{\kappa-\lambda}n^{1/2}}{(M_1z)^{1/2}} + \frac{p^{2\kappa-2\lambda}n}{2M_1z} + \sum_{j\geq3} c_j \left(  \frac{p^{\kappa-\lambda}n^{1/2}}{(M_1z)^{1/2}}\right)^j\right).
$$
Plugging \eqref{I(M_1z,n,q) after stationary phase} and  \eqref{h(M_1z,n,q)} into \eqref{eqn:I(n)definition}, we get
\begin{equation}
  \mathfrak{I}(m,n_1,n_2) = t^{-1}\int_{\mathbb{R}}   W_2( z)
  e\left(h_2\Big(M_1z,n_1,q\Big)
  - h_2\Big(M_1z,n_2,q\Big)
  - \frac{mM_1}{p^{\kappa}d} z \right) \dd z,
\end{equation}
where $W_2$ is an 1-inert function.

Then we deal with $z$-integral.
\begin{lemma}\label{lemma:I}
  We have
 \begin{enumerate}
    \item  If $m\gg \frac{N^2 X d}{R^3 Q t^2p^{\kappa+\lambda}} N^\varepsilon$, then $\mathfrak{I}(m,n_1,n_2) \ll N^{-A}$.
    \item  If $\frac{dN}{R^2p^{\kappa}t^2} N^\varepsilon \ll m \ll \frac{N^2 X d}{R^3 Q t^2p^{\kappa+\lambda}} N^\varepsilon$, then
        $\mathfrak{I}(m,n_1,n_2) \ll N^{-A}$ unless
        $|n_1-n_2| \frac{Q}{p^{\lambda}RX} \asymp  \frac{|m|R^2p^{\kappa}t^2}{Nd}$
        in which case we have
        $$
          \mathfrak{I}(m,n_1,n_2) \ll \left(\frac{dN}{|m|R^2p^{\kappa}t^4}\right)^{1/2}.
        $$
    \item  If $m \ll \frac{dN}{R^2p^{\kappa}t^2} N^\varepsilon $, then
        $\mathfrak{I}(m,n_1,n_2) \ll N^{-A}$ unless
        $$
          |n_1-n_2| \ll \frac{p^{\lambda}RX}{Q} N^\varepsilon,
        $$
        in which case we have $\mathfrak{I}(m,n_1,n_2) \ll t^{-1}$.
  \end{enumerate}
\end{lemma}

\begin{proof}
  (1).
  Let
  $$
    \rho(z) = h_2\Big(M_1z,n_1,q\Big)
  - h_2\Big(M_1z,n_2,q\Big)
  - \frac{mM_1}{p^{\kappa}d} z .
  $$
  Note that
\begin{multline*}
  h_2\Big(M_1z,n_1,q\Big)
  - h_2\Big(M_1z,n_2,q\Big) \\
    =
    \frac{t}{\pi}\left( \frac{p^{\kappa-\lambda}}{(M_1z)^{1/2}}\left(n_1^{1/2}-n_2^{1/2}\right) + \frac{p^{2\kappa-2\lambda}}{2M_1z}\left(n_1-n_2\right) + \sum_{j\geq3} c_j \left(  \frac{p^{\kappa-\lambda}}{(M_1z)^{1/2}}\right)^j\left(n_1^{j/2}-n_2^{j/2}\right)\right).
  \end{multline*}
For $n_1\asymp n_2 \asymp NX^2/Q^2$, we have
$$
    \frac{\partial^k }{\partial z^k} \left(  h_2\Big(M_1z,n_1,q\Big)
  - h_2\Big(M_1z,n_2,q\Big)\right) \ll_k  \frac{NX}{p^{\lambda}RQ} , \quad
    k\geq1.
$$
  By repeated integration by parts, we have $\mathfrak{I}(m,n_1,n_2)$ is negligibly small unless $m\ll \frac{N^2 X d}{R^3 Q t^2p^{\kappa+\lambda}} N^\varepsilon$.

  (2). Note that
  $$
  \begin{aligned}
  \left( \frac{p^{\kappa-\lambda}}{(M_1z)^{1/2}}\right)^j\left(n_1^{j/2}-n_2^{j/2}\right)
  &=\frac{p^{\kappa-\lambda}}{(M_1z)^{1/2}}\left(n_1^{1/2}-n_2^{1/2}\right)
  \left( \frac{p^{\kappa-\lambda}}{(M_1z)^{1/2}}\right)^{j-1}
  \sum_{i+\ell=j-1}n_1^{i/2}n_2^{\ell/2}\\
  &\asymp \frac{p^{\kappa-\lambda}}{(M_1z)^{1/2}}\left(n_1^{1/2}-n_2^{1/2}\right)
  \left(\frac{N^{1/2}}{tRp^{\lambda}}\right)^{j-1}\left(\frac{N^{1/2}X}{Q}\right)^{j-1}\\
  &\ll \frac{p^{\kappa-\lambda}}{(M_1z)^{1/2}}\left(n_1^{1/2}-n_2^{1/2}\right)t^{-(j-1)\varepsilon},
  \end{aligned}
  $$
  for $j \geq 2$.
  We have
  $$
    \frac{\partial^k }{\partial z^k} \left( h_2\Big(M_1z,n_1,q\Big)
  - h_2\Big(M_1z,n_2,q\Big) \right)  \asymp_k
   |n_1^{1/2}-n_2^{1/2}| \frac{N^{1/2}}{p^{\lambda}R} \asymp |n_1-n_2| \frac{Q}{p^{\lambda}RX},
  $$
  for any $ k\geq1$.\\
By repeated integration by parts, we have $\mathfrak{I}(m,n_1,n_2)$ is negligibly small unless $|n_1-n_2| \frac{Q}{p^{\lambda}RX} \asymp  \frac{|m|R^2p^{\kappa}t^2}{Nd}$.
Note that
   $$
    \rho^{''}(z) \asymp |n_1-n_2| \frac{Q}{p^{\lambda}RX} \asymp  \frac{|m|R^2p^{\kappa}t^2}{Nd}.
   $$
 By Lemma \ref{lem: rth derivative test, dim 1} with $r=2$, we have
  $$
    \mathfrak{I}(m,n_1,n_2) \ll t^{-1}\left( \frac{|m|R^2p^{\kappa}t^2}{Nd}\right)^{-1/2} = \left(\frac{dN}{|m|R^2p^{\kappa}t^4}\right)^{1/2}.
  $$

  (3). For $m \ll \frac{dN}{R^2p^{\kappa}t^2}N^\varepsilon$, by repeated integration by parts, $\mathfrak{I}(m,n_1,n_2)$ is negligibly small unless $|n_1-n_2| \ll \frac{p^{\lambda}RX}{Q} N^\varepsilon$,
  in which case $\mathfrak{I}(m,n_1,n_2) \ll t^{-1}$.
\end{proof}

\subsection{The small modulo case}
Assume that $\frac{NX}{p^{\lambda}RQ}\gg t^{1-\varepsilon}$. Then we have
$R \ll \frac{NX}{p^{\lambda}Qt^{1-\varepsilon}}.$
Recall that
\begin{equation*}
  \mathfrak{I}(m,n_1,n_2) = \int_{\mathbb{R}}   W_1\left( z\right)
  I(M_1z,n_1,q)\overline{I(M_1z,n_2,q)}e\left(-\frac{mM_1}{p^{\kappa}d} z \right)  \dd z
\end{equation*}
with
\begin{equation*}
I(M_1z,n,q)= \int_{0}^{\infty}V_3(\xi)e\left(-\frac{t\log \xi}{\pi}+ 2\frac{N^{1/2}n^{1/2}}{qp^{\lambda}}\xi
+2\sigma_2\frac{N^{1/2}M_1^{1/2}z^{1/2}}{qp^{\kappa}}\xi\right)\mathrm{d}\xi.
\end{equation*}

\begin{lemma}
\begin{enumerate}
\item For $m \gg \frac{N^{1/2}d}{M_1^{1/2}q}N^{\varepsilon}$, we have
$\mathfrak{I}(m,n_1,n_2)\ll N^{-A}$.
\item For $\frac{p^{\kappa}d}{M_1}N^{\varepsilon} \ll m \ll \frac{N^{1/2}d}{M_1^{1/2}q}N^{\varepsilon}$, we have
$\mathfrak{I}(m,n_1,n_2)\ll \frac{1}{t}\left(\frac{mNM_1^2}{q^2p^{3\kappa}t^2d}\right)^{-1/3}$.
\item For $m \ll \frac{p^{\kappa}d}{M_1}N^{\varepsilon}$, we have $\mathfrak{I}(m,n_1,n_2)\ll
\frac{1}{t}.$
\end{enumerate}
\end{lemma}
\begin{proof}
  (1).
Plugging \eqref{I(M_1z,n,q)definition} into \eqref{eqn:I(n)definition} and interchanging the order of integrals, we express $\mathfrak{I}(m,n_1,n_2)$ as
\begin{equation}\label{I small modulo}
\begin{aligned}
&\int_{0}^{\infty}\int_{0}^{\infty}V_3(\xi_1)\overline{V_3(\xi_2)}e\left(-\frac{t\log \xi_1}{\pi}+\frac{t\log \xi_2}{\pi}+2\frac{N^{1/2}n_1^{1/2}}{qp^{\lambda}}\xi_1 - \frac{N^{1/2}n_2^{1/2}}{qp^{\lambda}}\xi_2
  \right)\\
&\cdot\left(\int_{\mathbb{R}}  W_1\left( z\right)
e\left( 2\sigma_2\frac{N^{1/2}\left(M_1z\right)^{1/2}}{qp^{\kappa}}\xi_1
-2\sigma_2\frac{N^{1/2}\left(M_1z\right)^{1/2}}{qp^{\kappa}}\xi_2-\frac{mM_1}{p^{\kappa}d} z \right)  \mathrm{d}z \right)\mathrm{d}\xi_1\mathrm{d}\xi_2.
\end{aligned}
\end{equation}
We first deal with the $z$-integral in \eqref{I small modulo}. Let
$$
\rho_2(z)= 2\sigma_2 \frac{N^{1/2}\left(M_1z\right)^{1/2}}{qp^{\kappa}}\xi_1
- 2\sigma_2\frac{N^{1/2}\left(M_1z\right)^{1/2}}{qp^{\kappa}}\xi_2-\frac{mM_1}{p^{\kappa}d} z.
$$
Note that
$$
    \frac{\partial^k }{\partial z^k} \left(2 \sigma_2 \frac{N^{1/2}\left(M_1z\right)^{1/2}}{qp^{\kappa}}\xi_1 -2\sigma_2 \frac{N^{1/2}\left(M_1z\right)^{1/2}}{qp^{\kappa}}\xi_2\right) \ll_k  \frac{N^{1/2}\left(M_1\right)^{1/2}}{qp^{\kappa}} , \quad  k\geq1.
$$
By repeated integration by parts, we have the $z$-integral is negligibly small unless $m\ll \frac{N^{1/2}d}{M_1^{1/2}q}N^{\varepsilon}$.

  (2). Now we consider the case $m\ll \frac{N^{1/2}d}{M_1^{1/2}q}N^{\varepsilon}$.
For $i=1,2$, we have
$$
    I(M_1z,n_i,q)=\int_{0}^{\infty}V_3(\xi_i)e\left(-\frac{t\log \xi_i}{\pi} +2\frac{N^{1/2}n_i^{1/2}}{qp^{\lambda}}\xi_i
  + 2 \sigma_2 \frac{N^{1/2}\left(M_1z\right)^{1/2}}{qp^{\kappa}}\xi_i\right)\mathrm{d}\xi_i.
$$
For $I(M_1z,n_i,q)$, it is negligibly small by repeated integration by parts unless
$t \asymp \frac{N^{1/2}n_i^{1/2}}{qp^{\lambda}}+\sigma_2\frac{N^{1/2}\left(M_1z\right)^{1/2}}{qp^{\kappa}}$, in which case we apply the stationary phase method to the $\xi_i$-integral.
Let $h_4(\xi_i)= -\frac{t\log \xi_i}{\pi} + 2\frac{N^{1/2}n_i^{1/2}}{qp^{\lambda}}\xi_i
+ 2\sigma_2\frac{N^{1/2}\left(M_1z\right)^{1/2}}{qp^{\kappa}}\xi_i$.
Then we have
$$
  h_4'(\xi_i) = -\frac{t}{\pi \xi_i}+ 2\frac{N^{1/2}n_i^{1/2}}{qp^{\lambda}}
+ 2\sigma_2 \frac{N^{1/2}\left(M_1z\right)^{1/2}}{qp^{\kappa}}, \quad
  h_4''(\xi_i) =  \frac{t}{\pi \xi_i^2}  \asymp t.
$$
Then we have the stationary point
$$\xi_{i0}= \frac{t}{2\pi \left(\frac{N^{1/2}n_i^{1/2}}{qp^{\lambda}}
+\sigma_2\frac{N^{1/2}\left(M_1z\right)^{1/2}}{qp^{\kappa}}\right)} \asymp 1,
$$
and
$$
h_4(\xi_{i0})= \frac{t}{\pi}\log\frac{2 e\pi
\left(\frac{N^{1/2}n_i^{1/2}}{qp^{\lambda}}+\sigma_2\frac{N^{1/2}\left(M_1z\right)^{1/2}}{qp^{\kappa}}\right)}{t}.
$$
Hence, we have
\begin{equation}\label{I(M_1z,n_i,q)after Stationary Phase small case }
I(M_1z,n_i,q)=t^{-\frac{1}{2}}
V_{1i}\left(\frac{\frac{N^{1/2}n_i^{1/2}}{qp^{\lambda}}
+\sigma_2\frac{N^{1/2}\left(M_1z\right)^{1/2}}{qp^{\kappa}}}{t/2\pi}\right)
e\left(\frac{t}{\pi}\log\frac{2 e\pi
\left(\frac{N^{1/2}n_i^{1/2}}{qp^{\lambda}}
+\sigma_2\frac{N^{1/2}\left(M_1z\right)^{1/2}}{qp^{\kappa}}\right)}{t}\right),
\end{equation}
where $V_{1i}$ are $1$-inert functions supported on $\xi_{i0} \asymp 1$.\\
We insert \eqref{I(M_1z,n_i,q)after Stationary Phase small case } into \eqref{eqn:I(n)definition}, make a change of variable $z\rightarrow z^2$ and write
$a(n,z):=\frac{N^{1/2}n^{1/2}}{qp^{\lambda}}
+\sigma_2\frac{N^{1/2}M_1^{1/2}z}{qp^{\kappa}}$, then we obtain
\begin{equation}\label{I(m,n_1,n_2)6.10}
\begin{aligned}
\mathfrak{I}(m,n_1,n_2)= & \frac{2}{t}\int_{\mathbb{R}} z W_1\left(z^2\right)
V_{11}\left(\frac{a(n_1,z)}{t/2\pi}\right) \overline{V_{12}\left(\frac{a(n_2,z)}{t/2\pi}\right)}e\left(\rho_3(z)\right)  \dd z,
\end{aligned}
\end{equation}
where
$$
\rho_3(z)=\frac{t}{\pi}\log a(n_1,z)
-\frac{t}{\pi}\log a(n_2,z)
-\frac{mM_1}{p^{\kappa}d} z^2.
$$
Note that
\begin{equation*}
\frac{\partial }{\partial z} \left(\frac{t}{\pi}\log a(n_1,z)
-\frac{t}{\pi}\log a(n_2,z)\right)
=\frac{t}{\pi}\frac{\sigma_2N^{1/2}M_1^{1/2} }{qp^{\kappa}}
\left(\frac{1}{a(n_1,z)}-\frac{1}{a(n_2,z)}\right).
\end{equation*}
For $\frac{p^{\kappa}d}{M_1}N^{\varepsilon} \ll m \ll \frac{N^{1/2}d}{M_1^{1/2}q}N^{\varepsilon}$,
by repeated integration by parts, $\mathfrak{I}(m,n_1,n_2)$ is negligible small unless
\begin{equation}\label{condition of order}
\frac{tN^{1/2}M_1^{1/2}}{qp^{\kappa}}
\left(\frac{1}{a(n_1,z)}-\frac{1}{a(n_2,z)}\right)
\asymp \frac{mM_1}{p^{\kappa}d},
\end{equation}
in which case we consider the third derivative of $\rho_3(z)$
\begin{equation}\label{third derivative}
\begin{aligned}
    \frac{\partial ^3}{\partial z^3}\rho_3(z)
    =&\frac{2t}{\pi}\left(\frac{\sigma_2N^{1/2}M_1^{1/2}}{qp^{\kappa}}\right)^3
    \left(\frac{1}{a(n_1,z)^3}-\frac{1}{a(n_2,z)^3}\right).
\end{aligned}
\end{equation}
Note that
$\frac{1}{a^2}+\frac{1}{ab}+\frac{1}{b^2}\gg \min(\frac{1}{a^2},\frac{1}{b^2})$ for any $a,b\neq0$.
Then we have
\begin{equation}\label{quadratic term}
\begin{aligned}
\frac{1}{a(n_1,z)^3}-\frac{1}{a(n_2,z)^3}&=\left(\frac{1}{a(n_1,z)}-\frac{1}{a(n_2,z)}\right)
\left(\frac{1}{a(n_1,z)^2}+\frac{1}{a(n_1,z)a(n_2,z)}+\frac{1}{a(n_2,z)^2}\right)\\
& \gg \left(\frac{1}{a(n_1,z)}-\frac{1}{a(n_2,z)}\right) \frac{1}{t^2}.
\end{aligned}
\end{equation}
Therefore, by \eqref{condition of order}, \eqref{third derivative} and \eqref{quadratic term} we obtain
\begin{equation*}
\begin{aligned}
    \frac{\partial ^3}{\partial z^3}\rho_3(z)
     \gg t\left(\frac{N^{1/2}M_1^{1/2}}{qp^{\kappa}}\right)^3
    \left(\frac{1}{a(n_1,z)}-\frac{1}{a(n_2,z)}\right)\frac{1}{t^2}
    \asymp \frac{mM_1}{p^{\kappa}d}\left(\frac{N^{1/2}M_1^{1/2}}{qp^{\kappa}t}\right)^2
    =\frac{mNM_1^2}{q^2p^{3\kappa}t^2d}.
    \end{aligned}
\end{equation*}
By Lemma \ref{lem: rth derivative test, dim 1} with $r=3$, we have
$$
z\textrm{-integral}\ll \left(\frac{mNM_1^2}{q^2p^{3\kappa}t^2d}\right)^{-1/3}.
$$
Hence,
$$
\mathfrak{I}(m,n_1,n_2)\ll \frac{1}{t}\left(\frac{mNM_1^2}{q^2p^{3\kappa}t^2d}\right)^{-1/3}.
$$

(3). For $m \ll \frac{p^{\kappa}d}{M_1}N^{\varepsilon}$, we estimate \eqref{I(m,n_1,n_2)6.10} trivially and obtain
$\mathfrak{I}(m,n_1,n_2)\ll\frac{1}{t}.$

\end{proof}

\section{The zero frequency}
In this section, we deal with the terms with $m=0$. The character sum with $m=0$ is
$$
\mathfrak{C}(0,n_1,n_2)=p^{2\kappa}\sum_{d|(n_1-n_2,p^{\lambda})}d\mu\left(p^{\lambda}/d\right).
$$
The contributions of the zero frequency to $\Sigma(d,q)$ and $S^{\flat}(N)$ are denoted by $\Sigma_0(d,q)$ and $S^{\flat}_0(N)$, respectively.
\subsection{Large modulo case: $N^{\varepsilon} \ll \frac{NX}{p^{\lambda}RQ}\ll t^{1-\varepsilon}$}
In this case, we have
$|n_1-n_2| \ll \frac{p^{\lambda}RX}{Q} N^\varepsilon$, $\mathfrak{I}(0,n_1,n_2) \ll \frac{1}{t}$.
The contributions of this part to $\Sigma(d,q)$ and $S^{\flat}(N)$ are denoted by $\Sigma_{01}(d,q)$ and $S^{\flat}_{01}(N)$, respectively.
By \eqref{Sigma(d,q)}, \eqref{K after Poisson} and $M_1\asymp \frac{t^2R^2p^{2\kappa}}{N}$, we get
\begin{equation}\label{Sigma_01(d,q)}
\begin{aligned}
  \Sigma_{01}(d,q) &\ll  \frac{Q^4}{N^2 X^4}\frac{R^2p^{\kappa}t^2}{Nd}
  \sum_{n_1\asymp \frac{NX^2}{Q^2} } |\lambda_g(n_1)|^2
  \left(\sum_{\substack{n_2\asymp \frac{NX^2}{Q^2} \\ n_2\equiv n_1 \bmod dp^{\lambda} \\ |n_1-n_2| \ll \frac{p^{\lambda}RX}{Q} N^{\varepsilon}}}
   \frac{p^{2\kappa+\lambda}}{t}
   + \sum_{\substack{n_2\asymp \frac{NX^2}{Q^2} \\ n_2\equiv n_1 \bmod dp^{\lambda-1} \\ |n_1-n_2| \ll \frac{p^{\lambda}RX}{Q} N^{\varepsilon}}}
    \frac{p^{2\kappa+\lambda-1}}{t}\right)\\
  &\ll \frac{R^2p^{\kappa}t^2}{Nd}\frac{Q^2}{N X^2}N^{\varepsilon}
  \left(1+\frac{p^{\lambda}RX}{Qdp^{\lambda}}\right)\frac{p^{2\kappa+\lambda}}{t}\\
  &\ll \frac{R^2Q^2tN^{\varepsilon}}{N^2 X^2d}p^{3\kappa+\lambda}
  +\frac{R^3QtN^{\varepsilon}}{N^{2} Xd^2}p^{3\kappa+\lambda}.
\end{aligned}
\end{equation}
By \eqref{Sb(N) after CAUCHY}, \eqref{Sigma_01(d,q)} and $N\ll (p^{\kappa}t)^{2+\varepsilon}$, we obtain
\begin{equation*}
\begin{aligned}
  S^{\flat}_{01}(N)
  &\ll \frac{NX^{3/2}R^{1/2}t^{1/2}}{Q^{3/2}p^{(\kappa+\lambda)/2}}\sum_{\substack{q\sim R\\ (q,p)=1}} \frac{1}{q}\sum_{d|q}d
  \left(\frac{R^2Q^2tN^{\varepsilon}}{N^2 X^2d}p^{3\kappa+\lambda}
  +\frac{R^3QtN^{\varepsilon}}{N^{2} Xd^2}p^{3\kappa+\lambda}\right)^{1/2}\\
  &\ll N^{1/2+\varepsilon}\left(p^{3\kappa/2-3\lambda/4}\frac{t^{3/2}}{K^{3/4}}
  +p^{\kappa-\lambda/2}\frac{t}{K^{1/2}}\right).
\end{aligned}
\end{equation*}
\subsection{Small modulo case: $\frac{NX}{p^{\lambda}RQ}\gg t^{1-\varepsilon}$}
We have $\mathfrak{I}(0,n_1,n_2) \ll \frac{1}{t}$, $M_1\ll \frac{NX^2p^{2\kappa-2\lambda}}{Q^2}$ and
$R\ll \frac{NX}{p^{\lambda}Qt^{1-\varepsilon}}$.
The contributions of this part to $\Sigma(d,q)$ and $S^{\flat}(N)$ are denoted by $\Sigma_{02}(d,q)$ and $S^{\flat}_{02}(N)$, respectively.
By \eqref{Sigma(d,q)} and \eqref{K after Poisson}, we get
\begin{equation}\label{Sigma_02(d,q)}
\begin{aligned}
  \Sigma_{02}(d,q) &\ll  \frac{Q^4}{N^2 X^4}\frac{M_1}{p^{\kappa}d}
  \sum_{n_1\asymp \frac{NX^2}{Q^2} } |\lambda_g(n_1)|^2
  \left(\sum_{\substack{n_2\asymp \frac{NX^2}{Q^2} \\ n_2\equiv n_1 \bmod dp^{\lambda} }}   \frac{p^{2\kappa+\lambda}}{t}
  + \sum_{\substack{n_2\asymp \frac{NX^2}{Q^2} \\ n_2\equiv n_1 \bmod dp^{\lambda-1} }}   \frac{p^{2\kappa+\lambda-1}}{t}\right)\\
  &\ll \frac{M_1}{p^{\kappa}d}\frac{Q^2}{N X^2}
  \left(1+\frac{NX^2}{Q^2dp^{\lambda}}\right)\frac{p^{2\kappa+\lambda}}{t}\\
  &\ll \frac{Q^2M_1N^{\varepsilon}}{N X^2dt}p^{\kappa+\lambda}
  +\frac{M_1N^{\varepsilon}}{d^2t}p^{\kappa}.
\end{aligned}
\end{equation}
By \eqref{Sb(N) after CAUCHY}, \eqref{Sigma_02(d,q)} and $N\ll (p^{\kappa}t)^{2+\varepsilon}$, we obtain
\begin{equation*}
\begin{aligned}
  S^{\flat}_{02}(N)&= \sup_{M_1\ll \frac{NX^2p^{2\kappa-2\lambda}}{Q^2}}\frac{N^{5/4}X^{3/2}M_1^{1/4}}{Q^{3/2}p^{(2\kappa+\lambda)/2}}\sum_{\substack{q\sim R\\ (q,p)=1}} \frac{1}{q}\sum_{d|q}d \cdot \Sigma_0^{1/2}\\
  &\ll \sup_{M_1\ll \frac{NX^2p^{2\kappa-2\lambda}}{Q^2}}\frac{N^{5/4}X^{3/2}M_1^{1/4}N^{\varepsilon}}{Q^{3/2}p^{(2\kappa+\lambda)/2}}\sum_{\substack{q\sim R\\ (q,p)=1}} \frac{1}{q}\sum_{d|q}d
  \left(\frac{Q^2M_1N^{\varepsilon}}{N X^2dt}p^{\kappa+\lambda}
  +\frac{M_1N^{\varepsilon}}{d^2t}p^{\kappa}\right)^{1/2}\\
  &\ll N^{1/2+\varepsilon}\left(p^{3\kappa/2-3\lambda/4}\frac{K^{5/4}}{t^{1/2}}+p^{\kappa-\lambda/2}\frac{K^{3/2}}{t^{1/2}}\right).
\end{aligned}
\end{equation*}

\section{The non-zero frequencies}
In this section, we deal with the terms with $m\neq 0$.
\subsection{Large modulo case: $N^{\varepsilon} \ll \frac{NX}{p^{\lambda}RQ}\ll t^{1-\varepsilon}$}
\subsubsection{$m \ll \frac{dN}{R^2p^{\kappa}t^2} N^{\varepsilon}$}
For $m \ll \frac{dN}{R^2p^{\kappa}t^2} N^{\varepsilon}$ and $m \neq 0$, we have
$|n_1-n_2| \ll \frac{p^{\lambda}RX}{Q} N^{\varepsilon}$
and $\mathfrak{I}(m,n_1,n_2) \ll t^{-1}$.
The contributions of this part to $\Sigma(d,q)$ and $S^{\flat}(N)$ are denoted by $\Sigma_{11}(d,q)$ and $S^{\flat}_{11}(N)$, respectively.
By \eqref{Sigma(d,q)}, \eqref{K after Poisson} and Lemma \ref{character sum estimate}(1)(2), we get
\begin{equation}\label{Sigma_{11}(d,q)}
\begin{aligned}
   \Sigma_{11}(d,q) & \ll \frac{Q^4}{N^2 X^4}\frac{R^2p^{\kappa}t^2}{Nd}
  \sum_{n_1\asymp \frac{NX^2}{Q^2} } |\lambda_g(n_1)|^2
  \cdot \bigg(\sum_{\substack{n_2\asymp \frac{NX^2}{Q^2} \\ n_2\equiv n_1 \bmod dp^{\beta-\kappa+\lambda} \\ |n_1-n_2| \ll \frac{p^{\lambda}RX}{Q} N^\varepsilon}}
  \sum_{\substack{m\ll \frac{dN}{R^2p^{\kappa}t^2} N^\varepsilon \\ p^{\kappa-\lambda+\beta-\kappa+\lambda}\mid m }}
  \frac{p^{2\kappa+\lambda+\delta_1}}{t}\\
  &\quad \quad \quad \quad \quad \quad \quad +\sum_{0\leq \ell \leq\lambda-\kappa/2 }
  \sum_{\substack{n_2\asymp \frac{NX^2}{Q^2} \\ n_2\equiv n_1 \bmod dp^{\ell},~p^{\ell}\| n_1-n_2 \\ |n_1-n_2| \ll \frac{p^{\lambda}RX}{Q} N^\varepsilon}}
  \sum_{\substack{m\ll \frac{dN}{R^2p^{\kappa}t^2} N^\varepsilon \\ p^{\kappa-\lambda+\ell}\mid m }}
  \frac{p^{5\kappa/2+2\ell+\delta_1+\delta_2/2}}{t}\bigg)\\
  &\ll\frac{R^2Q^2t^2p^{\kappa}}{N^2X^2d}N^{\varepsilon}
  \left(\left(1+\frac{p^{\lambda}RX}{Qdp^{\beta-\kappa+\lambda}}\right)
  \frac{Nd \cdot p^{2\kappa+\lambda+\delta_1}}{R^2p^{\kappa}t^2p^{\beta}\cdot t}
  +\frac{p^{\lambda}RX}{Qdp^{\ell}}
  \frac{Nd\cdot p^{5\kappa/2+2\ell+\delta_1+\delta_2/2}}{R^2p^{\kappa}t^2p^{\kappa-\lambda+\ell}\cdot t}\right)\\
  &\ll\frac{Q^2N^{\varepsilon}}{NX^2t}p^{3\kappa/2+\lambda+\delta_1+\delta_2/2}
  +\frac{QRN^{\varepsilon}}{NXdt}p^{2\kappa+\lambda+\delta_1+\delta_2}
  +\frac{QRN^{\varepsilon}}{NXdt}p^{3\kappa/2+2\lambda+\delta_1+\delta_2/2}.
\end{aligned}
\end{equation}
By \eqref{Sb(N) after CAUCHY}, \eqref{Sigma_{11}(d,q)} and $N\ll (p^{\kappa}t)^{2+\varepsilon}$, we obtain
\begin{equation*}
\begin{aligned}
  S^{\flat}_{11}(N)&=
  \frac{NX^{3/2}R^{1/2}t^{1/2}}{Q^{3/2}p^{(\kappa+\lambda)/2}}\sum_{\substack{q\sim R\\ (q,p)=1}} \frac{1}{q}\sum_{d|q}d \cdot \Sigma_{11}(d,q)^{1/2}\\
  &\ll N^{1/2+\varepsilon}\left(p^{5\kappa/4-\lambda/2+3/4}\frac{t}{K^{1/2}}
  +p^{\kappa-\lambda/4+1}\frac{t^{1/2}}{K^{1/2}}
  +p^{3\kappa/4+\lambda/4+3/4}\frac{t^{1/2}}{K^{1/4}}\right).
\end{aligned}
\end{equation*}
\subsubsection{$\frac{dN}{R^2p^{\kappa}t^2} N^\varepsilon \ll m \ll \frac{N^2 X d}{R^3 Q t^2p^{\kappa+\lambda}} N^\varepsilon$}
For $m \sim M$ with $\frac{dN}{R^2p^{\kappa}t^2} N^\varepsilon \ll M \ll \frac{N^2 X d}{R^3 Q t^2p^{\kappa+\lambda}} N^\varepsilon$, we have $|n_1-n_2| \asymp \frac{MR^3Xt^2p^{\kappa+\lambda}}{NQd}$ and $\mathfrak{I}(m,n_1,n_2) \ll \left(\frac{dN}{MR^2p^{\kappa}t^4}\right)^{1/2}$.
The contributions of this part to $\Sigma(d,q)$ and $S^{\flat}(N)$ are denoted by $\Sigma_{12}(d,q)$ and $S^{\flat}_{12}(N)$, respectively.
\begin{equation}\label{Sigma_{12}(d,q)}
\begin{aligned}
  \Sigma_{12}(d,q) \ll  &\frac{Q^4}{N^2 X^4}\frac{R^2p^{\kappa}t^2}{Nd}
  \sum_{n_1\asymp \frac{NX^2}{Q^2} } |\lambda_g(n_1)|^2
  \bigg(\sum_{\substack{n_2\asymp \frac{NX^2}{Q^2} \\ n_2\equiv n_1 \bmod dp^{\beta-\kappa+\lambda} \\ |n_1-n_2| \ll \frac{MR^3Xt^2p^{\kappa+\lambda}}{NQd}}}
  \sum_{\substack{m\sim M \\ p^{\kappa-\lambda+\beta-\kappa+\lambda}\mid m }}
  \frac{p^{2\kappa+\lambda+\delta_1}d^{1/2}N^{1/2}}{M^{1/2}Rp^{\kappa/2}t^2}\\
  & \quad \quad \quad \quad \quad \quad
  +\sum_{0\leq \ell \leq\lambda-\kappa/2 }
  \sum_{\substack{n_2\asymp \frac{NX^2}{Q^2} \\ n_2\equiv n_1 \bmod dp^{\ell},~p^{\ell}\| n_1-n_2 \\ |n_1-n_2| \ll \frac{MR^3Xt^2p^{\kappa+\lambda}}{NQd}}}
  \sum_{\substack{m\sim M \\ p^{\kappa-\lambda+\ell}\mid m }}
  \frac{p^{5\kappa/2+2\ell+\delta_1+\delta_2/2}d^{1/2}N^{1/2}}{M^{1/2}Rp^{\kappa/2}t^2}\bigg)\\
 \ll& \frac{RQ^2p^{\kappa/2}N^{\varepsilon}}{M^{1/2}N^{3/2}X^2d^{1/2}}
  \left(\left(1+\frac{MR^3Xt^2p^{\kappa+\lambda}}{NQd^2p^{\beta-\kappa+\lambda}}\right)
  \frac{Mp^{2\kappa+\lambda
  +\delta_1}}{p^{\beta}}
  +\sup_{\ell}\frac{MR^3Xt^2p^{\kappa+\lambda}}{NQd^2p^{\ell}}
  \frac{Mp^{5\kappa/2+2\ell+3/2}}{p^{\kappa-\lambda+\ell}}\right)\\
  \ll&\frac{Q^{3/2}p^{3\kappa/2+\lambda/2+3/2}N^{\varepsilon}}{R^{1/2}X^{3/2}N^{1/2}t}
  +\frac{N^{1/2}X^{1/2}p^{2\kappa-\lambda/2+1}N^{\varepsilon}}{R^{1/2}Q^{1/2}dt}
  +\frac{N^{1/2}X^{1/2}p^{3\kappa/2+\lambda/2+3/2}N^{\varepsilon}}{R^{1/2}Q^{1/2}dt}.
\end{aligned}
\end{equation}
By \eqref{Sb(N) after CAUCHY}, \eqref{Sigma_{12}(d,q)} and $N\ll (p^{\kappa}t)^{2+\varepsilon}$, we obtain
$$
\begin{aligned}
  S^{\flat}_{12}(N)&= \frac{NX^{3/2}R^{1/2}t^{1/2}}{Q^{3/2}p^{(\kappa+\lambda)/2}}\sum_{\substack{q\sim R\\ (q,p)=1}} \frac{1}{q}\sum_{d|q}d \cdot \Sigma_{12}(d,q)^{1/2}\\
  &\ll N^{1/2+\varepsilon}\left(p^{5\kappa/4-\lambda/2+3/4}\frac{t}{K^{1/4}}
  +p^{\kappa-\lambda/4+1}t^{1/2}K^{1/2}
  +p^{3\kappa/4+\lambda/4+3/4}t^{1/2}K^{1/2}\right).
\end{aligned}
$$
\subsection{Small modulo case: $\frac{NX}{p^{\lambda}RQ}\gg t^{1-\varepsilon}$}
\subsubsection{$m \ll \frac{p^{\kappa}d}{M_1}N^{\varepsilon}$}
For $m \ll \frac{p^{\kappa}d}{M_1}N^{\varepsilon}$ and $m\neq0$, we have $\mathfrak{I}(m,n_1,n_2)\ll
\frac{1}{t}.$
The contributions of this part to $\Sigma(d,q)$ and $S^{\flat}(N)$ are denoted by $\Sigma_{13}(d,q)$ and $S^{\flat}_{13}(N)$, respectively.
By \eqref{Sigma(d,q)}, \eqref{K after Poisson} and Lemma \ref{character sum estimate}(1)(2), we get
\begin{equation}\label{Sigma_{13}(d,q)}
\begin{aligned}
 \Sigma_{13}(d,q) &\ll  \frac{Q^4}{N^2 X^4}\frac{M_1}{p^{\kappa}d}
  \sum_{n_1\asymp \frac{NX^2}{Q^2} } |\lambda_g(n_1)|^2
  \bigg(\sum_{\substack{n_2\asymp \frac{NX^2}{Q^2} \\ n_2\equiv n_1 \bmod dp^{\beta-\kappa+\lambda}}}
  \sum_{\substack{m\ll \frac{p^{\kappa}d}{M_1}N^{\varepsilon} \\ p^{\kappa-\lambda+\beta-\kappa+\lambda}\mid m }}
  p^{2\kappa+\lambda+\delta_1}\frac{1}{t}\\
  &\quad \quad \quad \quad \quad \quad \quad \quad \quad+\sum_{0\leq \ell \leq\lambda-\kappa/2 }
  \sum_{\substack{n_2\asymp \frac{NX^2}{Q^2} \\ n_2\equiv n_1 \bmod dp^{\ell}\\ p^{\ell}\| n_1-n_2}}
  \sum_{\substack{m\ll \frac{p^{\kappa}d}{M_1}N^{\varepsilon} \\ p^{\kappa-\lambda+\ell}\mid m }}
  p^{5\kappa/2+2\ell+\delta_1+\delta_2/2}\frac{1}{t}\bigg)\\
 & \ll \frac{Q^2}{NX^2}\frac{M_1}{p^{\kappa}d}
 \bigg(\left(1+\frac{NX^2}{Q^2dp^{\beta-\kappa+\lambda}}\right)
 \frac{\frac{p^{\kappa}d}{M_1}N^{\varepsilon}}{p^{\beta}}
 \frac{p^{2\kappa+\lambda+\delta_1}}{t}
 +\sup_{\ell}\frac{NX^2}{Q^2dp^{\ell}}
 \frac{\frac{p^{\kappa}d}{M_1}N^{\varepsilon}}{p^{\kappa-\lambda+\ell}}
  \frac{p^{5\kappa/2+2\ell+3/2}}{t}\bigg)\\
  & \ll \frac{Q^2p^{3\kappa/2+\lambda+3/2}N^{\varepsilon}}{NX^2t}
  +\frac{p^{2\kappa+2}N^{\varepsilon}}{dt}
  +\frac{p^{3\kappa/2+\lambda+3/2}N^{\varepsilon}}{dt}.
\end{aligned}
\end{equation}
By \eqref{Sb(N) after CAUCHY}, \eqref{Sigma_{13}(d,q)} and $N\ll (p^{\kappa}t)^{2+\varepsilon}$, we obtain
\begin{equation*}
\begin{aligned}
S^{\flat}_{13}(N)&= \sup_{M_1\ll \frac{NX^2p^{2\kappa-2\lambda}}{Q^2}}\frac{N^{5/4}X^{3/2}M_1^{1/4}}{Q^{3/2}p^{(2\kappa+\lambda)/2}}\sum_{\substack{q\sim R\\ (q,p)=1}} \frac{1}{q}\sum_{d|q}d \cdot \Sigma_{13}^{1/2}\\
  &\ll N^{1/2+\varepsilon}\left(p^{5\kappa/4-\lambda/2+3/4}\frac{K}{t^{1/2}}
  +p^{\kappa-\lambda/4+1}\frac{K^{5/4}}{t^{1/2}}
  +p^{3\kappa/4+\lambda/4+3/4}\frac{K^{5/4}}{t^{1/2}}\right).
\end{aligned}
\end{equation*}
\subsubsection{$\frac{p^{\kappa}d}{M_1}N^{\varepsilon} \ll m \ll \frac{N^{1/2}d}{M_1^{1/2}q}N^{\varepsilon}$}
For $\frac{p^{\kappa}d}{M_1}N^{\varepsilon} \ll m \ll \frac{N^{1/2}d}{M_1^{1/2}q}N^{\varepsilon}$, we have
$\mathfrak{I}(m,n_1,n_2)\ll \frac{1}{t}\left(\frac{mNM_1^2}{q^2p^{3\kappa}t^2d}\right)^{-1/3}$.
The contributions of this part to $\Sigma(d,q)$ and $S^{\flat}(N)$ are denoted by $\Sigma_{14}(d,q)$ and $S^{\flat}_{14}(N)$, respectively.
By \eqref{Sigma(d,q)}, \eqref{K after Poisson} and Lemma \ref{character sum estimate}(1)(2), we get
\begin{equation}\label{Sigma_{14}(d,q)}
\begin{aligned}
 \Sigma_{14}(d,q) &\ll  \frac{Q^4}{N^2 X^4}\frac{M_1}{p^{\kappa}d}
  \sum_{n_1\asymp \frac{NX^2}{Q^2} } |\lambda_g(n_1)|^2
  \bigg(\sum_{\substack{n_2\asymp \frac{NX^2}{Q^2} \\ n_2\equiv n_1 \bmod dp^{\beta-\kappa+\lambda}}}
  \sum_{\substack{m\ll \frac{N^{1/2}d}{M_1^{1/2}q}N^{\varepsilon} \\ p^{\kappa-\lambda+\beta-\kappa+\lambda}\mid m }}
  p^{2\kappa+\lambda+\delta_1}\frac{1}{t}\left(\frac{mNM_1^2}{q^2p^{3\kappa}t^2d}\right)^{-1/3}\\
  & \quad \quad \quad \quad \quad \quad
  +\sum_{0\leq \ell \leq\lambda-\kappa/2 }
  \sum_{\substack{n_2\asymp \frac{NX^2}{Q^2} \\ n_2\equiv n_1 \bmod dp^{\ell}\\ p^{\ell}\| n_1-n_2}}
  \sum_{\substack{m\ll \frac{N^{1/2}d}{M_1^{1/2}q}N^{\varepsilon} \\ p^{\kappa-\lambda+\ell}\mid m }}
  p^{5\kappa/2+2\ell+\delta_1+\delta_2/2}\frac{1}{t}\left(\frac{mNM_1^2}{q^2p^{3\kappa}t^2d}\right)^{-1/3}\bigg)\\
\ll&\frac{Q^2M_1^{1/3}q^{2/3}}{N^{4/3}X^2d^{2/3}t^{1/3}}
 \bigg(\left(1+\frac{NX^2}{Q^2dp^{\beta-\kappa+\lambda}}\right)
  \frac{\left(\frac{N^{1/2}d}{M_1^{1/2}q}N^{\varepsilon}\right)^{2/3}
  p^{2\kappa+\lambda
  +\delta_1}}{p^{\beta}}\\
  &\quad \quad \quad \quad \quad \quad \quad \quad\quad +\sup_{\ell}\frac{NX^2}{Q^2dp^{\ell}}
  \frac{\left(\frac{N^{1/2}d}{M_1^{1/2}q}N^{\varepsilon}\right)^{2/3}
  p^{5\kappa/2+2\ell+3/2}}{p^{\kappa-\lambda+\ell}}\bigg)\\
  \ll& \frac{Q^2p^{3\kappa/2+\lambda+3/2}N^{\varepsilon}}{NX^2t^{1/3}}
  +\frac{p^{2\kappa+2}N^{\varepsilon}}{dt^{1/3}}
  +\frac{p^{3\kappa/2+\lambda+3/2}N^{\varepsilon}}{dt^{1/3}}.
\end{aligned}
\end{equation}
By \eqref{Sb(N) after CAUCHY}, \eqref{Sigma_{14}(d,q)} and $N\ll (p^{\kappa}t)^{2+\varepsilon}$, we obtain
\begin{equation*}
\begin{aligned}
S^{\flat}_{14}(N)&= \sup_{M_1\ll \frac{NX^2p^{2\kappa-2\lambda}}{Q^2}}\frac{N^{5/4}X^{3/2}M_1^{1/4}}{Q^{3/2}p^{(2\kappa+\lambda)/2}}\sum_{\substack{q\sim R\\ (q,p)=1}} \frac{1}{q}\sum_{d|q}d \cdot \Sigma_{14}^{1/2}\\
  &\ll N^{1/2+\varepsilon}\left(p^{5\kappa/4-\lambda/2+3/4}\frac{K}{t^{1/6}}
  +p^{\kappa-\lambda/4+1}\frac{K^{5/4}}{t^{1/6}}
  +p^{3\kappa/4+\lambda/4+3/4}\frac{K^{5/4}}{t^{1/6}}\right).
\end{aligned}
\end{equation*}

\section{Conclusion}
By \eqref{the relation between S(N)and Sb(N)} and the upper bounds of $S^{\flat}_{01}(N)$, $S^{\flat}_{02}(N)$, $S^{\flat}_{11}(N)$, $S^{\flat}_{12}(N)$, $S^{\flat}_{13}(N)$,  $S^{\flat}_{14}(N)$, $S^{\flat}_{b0}(N)$, $S^{\flat}_{b1}(N)$, $S^{\flat}_{c0}(N)$, $S^{\flat}_{c1}(N)$, the best choice of parameters is $K=t^{\frac{4}{5}}$ and $\lambda=[\frac{3}{4}\kappa]$,  where $[x]$ denotes the largest integer which does not exceed $x$.
Therefore, we have shown the following bound
\begin{equation}\label{S(N)final bound}
S(N)\ll N^{1/2}
p^{3/4}(p^{\kappa})^{1-\frac{1}{16}+\varepsilon}t^{1-\frac{1}{10}+\varepsilon}.
\end{equation}
Substituting the estimate in \eqref{S(N)final bound} for $S(N)$ into \eqref{Lfunction after AFE},
we obtain
$$
  L\left(\frac{1}{2}+it,f\times g \times \chi\right) \ll_{\varepsilon} p^{\frac{3}{4}}(p^{\kappa})^{1-\frac{1}{16}+\varepsilon}t^{1-\frac{1}{10}+\varepsilon}.
$$
This completes the proof of Theorem \ref{main-theorem}.

\section{The complementary cases}\label{The complementary cases}

In this section we focus on the complementary cases whose estimations can be bounded by that of the main case. And we give the brief proofs for the complementary cases.
Next, we consider the following three cases:
$$
\textrm{case} ~ a \quad \frac{nN}{q^2 p^{2\lambda}}\ll N^{\varepsilon}, ~\frac{NX}{p^{\lambda}RQ}\ll N^{\varepsilon} ~ \textrm{and} ~
\frac{mN}{q^2p^{2\kappa}}\ll N^{\varepsilon},
$$
$$
\textrm{case} ~ b \quad \frac{nN}{q^2 p^{2\lambda}}\asymp \big(\frac{NX}{RQp^{\lambda}}\big)^2 \gg N^{\varepsilon} ~ \textrm{and} ~ \frac{mN}{q^2p^{2\kappa}}\ll N^{\varepsilon},
$$
$$
\textrm{case} ~ c \quad \frac{nN}{q^2 p^{2\lambda}}\ll N^{\varepsilon}, ~\frac{NX}{p^{\lambda}RQ}\ll N^{\varepsilon} ~ \textrm{and} ~
\frac{mN}{q^2p^{2\kappa}}\gg N^{\varepsilon}.
$$

It is easy to see that the contribution from case $a$ to $S(N)$ is negligibly small by repeated integration by parts for \eqref{psi2 case b}.

\subsection{Case b: $\frac{nN}{q^2 p^{2\lambda}}\asymp \big(\frac{NX}{RQp^{\lambda}}\big)^2 \gg N^{\varepsilon}$, $\frac{mN}{q^2p^{2\kappa}}\ll N^{\varepsilon}$}
We first use the same method as in \S 5 to $n$-sum and $u$-integral. Then we use the $\mathrm{GL_2}$ Voronoi summation formula to $m$-sum and obtain that
\begin{equation}\label{Sb(N) before S2 in 10}
\begin{aligned}
S^{\flat}_b(N)=
& \frac{X^{\frac{3}{2}}N^{\frac{1}{2}}}{Q^{\frac{3}{2}}}\sum_{\substack{q\sim R \\ (q,p)=1}} \frac{1}{(qp^{\lambda})^{1/2}} \sideset{}{^\star}\sum_{a\bmod{qp^{\lambda}}}
  \sum_{n\asymp \frac{NX^2}{Q^2}}\frac{\lambda_g(n)}{n}
e\left(\frac{\overline{a}n}{qp^{\lambda}}\right) \\
& \cdot \frac{1}{\tau(\overline{\chi})}\frac{N^{1-it}}{qp^{\kappa}}\sum_{c \mod p^{\kappa}}\overline{\chi}(c)\sum_{m}\lambda_f(m)
e\left(-\frac{\overline{a^{\ast}}m}{qp^{\kappa}}\right)
\Phi_2^{\pm}\left(\frac{mN}{q^2p^{2\kappa}}\right)+ O(N^{-A}).
\end{aligned}
\end{equation}

By Lemma \ref{Phi 2 m small}, $\Phi_2^{\pm}(y)$ is negligible small unless $t\asymp \frac{N^{1/2}n^{1/2}}{qp^{\lambda}}\asymp \frac{NX}{p^{\lambda}RQ}$, in which case
\begin{equation}\label{phi2 case b}
\Phi_2^{\pm}(y)=\int_{0}^{\infty}V(\xi)V_2\left(\frac{n Q^2}{NX^2\xi}\right)
e\left(-\frac{t\log \xi}{2\pi}+ 2\frac{N^{1/2}n^{1/2}}{qp^{\lambda}}\xi^{1/2}\right)J_f^{\pm}(y\xi)\mathrm{d}\xi,
\end{equation}
with $y^k {J_f^{\pm}}^{(k)}(y) \ll_{k, f} N^{\varepsilon}.$

Inserting \eqref{phi2 case b} into \eqref{Sb(N) before S2 in 10} and rearranging the order of summations and integrals, we arrive at
\begin{equation*}
\begin{aligned}
S^{\flat}_b(N)=& \frac{N^{3/2-it}X^{3/2}}{\tau(\overline{\chi})Q^{3/2}p^{(2\kappa+\lambda)/2}}\sum_{\substack{q\sim R\\ (q,p)=1}} \frac{1}{q^{\frac{3}{2}}}\sum_{d|q}d\mu\left(\frac{q}{d}\right)
\sum_{n\asymp\frac{NX^2}{Q^2}}\frac{\lambda_g(n)}{n}\\
&\cdot \sum_{\substack{m\ll \frac{q^2p^{2\kappa}N^{\varepsilon}}{N} \\ n\equiv mp^{\lambda}\overline{p^{2\kappa-\lambda}}\bmod d}}
\lambda_f(m)\mathfrak{C}^{\ast}(m,n,q)\Phi_2^{\pm}\left(\frac{mN}{q^2p^{2\kappa}}\right)\\
\end{aligned}
\end{equation*}
with
$$
\mathfrak{C}^{\ast}(m,n,q)= \sum_{c \mod p^{\kappa}}\overline{\chi}(c)
\sideset{}{^\star}\sum_{b\bmod{p^{\lambda}}}
e\left(\frac{-(\overline{ap^{\kappa-\lambda}+cq})m\bar{q}}{p^{\kappa}}+ \frac{n\bar{q}\bar{b}}{p^{\lambda}}\right).
$$
Applying the Cauchy--Schwarz inequality and the Poisson summation formula to $m$-sum, we have
\begin{equation}\label{Sb(N)caseb}
S^{\flat}_b(N)\ll \frac{NX^{3/2}}{Q^{3/2}p^{(\kappa+\lambda)/2}}\sum_{\substack{q\sim R\\ (q,p)=1}} \frac{1}{q^{\frac{1}{2}}}\sum_{d|q}d \cdot \Sigma_b(d,q)^{1/2},
\end{equation}
where
$$
  \Sigma_b(d,q) \ll  \frac{Q^4R^2p^{\kappa}}{N^3 X^4d}
  \sum_{n_1\asymp \frac{NX^2}{Q^2} } |\lambda_g(n_1)|^2
  \sum_{\substack{n_2\asymp \frac{NX^2}{Q^2} \\ n_2\equiv n_1 \bmod d}}
  \sum_{m\in\mathbb{Z}} |\mathfrak{C}(m,n_1,n_2)\cdot \mathfrak{I}_b(m,n_1,n_2)|,
$$
with
$$
  \mathfrak{C}(m,n_1,n_2) = \sum_{ \gamma \bmod p^{\kappa}}\mathfrak{C}^{\ast}(\gamma,n_1,q)\overline{\mathfrak{C}^{\ast}(\gamma,n_2,q)}
  e\left(\frac{m\left(n_1p^{3\kappa-\lambda}\overline{p^{\kappa+\lambda}}+\gamma d\bar{d}\right)}{dp^{\kappa}}\right)
$$
and
\begin{equation*}\label{eqn:I(n)1}
  \mathfrak{I}_b(m,n_1,n_2) := \int_{\mathbb{R}} V\left( \xi\right)
  \Phi_2^{\pm}\left( \xi\right)\overline{\Phi_2^{\pm}\left(\xi\right)}
  e\left(-\frac{mR^2p^{\kappa}}{Nd} \xi \right)  \dd \xi.
\end{equation*}
\begin{lemma}
\begin{enumerate}
\item If $m \gg \frac{N^{1+\varepsilon}d}{R^2p^{\kappa}}$, we have
$\mathfrak{I}(m,n_1,n_2)\ll N^{-A}$.
\item If $m \ll\frac{N^{1+\varepsilon}d}{R^2p^{\kappa}}$, we have
$\mathfrak{I}(m,n_1,n_2)\ll 1$.
\end{enumerate}
\end{lemma}
\begin{proof}
The first result can be done by applying integration by parts with respect to the $\xi$-integral. For $m\ll \frac{N^{1+\varepsilon}d}{R^2p^{\kappa}}$, we have the trivial bound $\mathfrak{I}(m,n_1,n_2)\ll 1$.
\end{proof}
The contribution from the zero frequency to $\Sigma_b(d,q)$ is
\begin{equation}\label{Sigma_{b0}(d,q)}
\Sigma_{b0}(d,q) \ll  \frac{Q^4R^2p^{\kappa}}{N^3 X^4d}
  \sum_{n_1\asymp \frac{NX^2}{Q^2} } |\lambda_g(n_1)|^2
  \left(\sum_{\substack{n_2\asymp \frac{NX^2}{Q^2} \\ n_2\equiv n_1 \bmod dp^{\lambda}}}   p^{2\kappa+\lambda} + \sum_{\substack{n_2\asymp \frac{NX^2}{Q^2} \\ n_2\equiv n_1 \bmod dp^{\lambda-1}}}   p^{2\kappa+\lambda-1}\right).
\end{equation}
Plugging \eqref{Sigma_{b0}(d,q)} into \eqref{Sb(N)caseb}, we get
\begin{equation*}
  S^{\flat}_{b0}(N)\ll N^{1/2+\varepsilon}\left(p^{3\kappa/2-3\lambda/4}\frac{K^{5/4}}{t^{3/2}}
  + p^{\kappa-\lambda/2}\frac{K^{3/2}}{t^{3/2}}\right).
\end{equation*}
The contribution from the non-zero frequencies to $\Sigma_b(d,q)$ is
\begin{equation}\label{Sigma_{b1}(d,q)}
\begin{aligned}
  \Sigma_{b1}(d,q) \ll  &\frac{Q^4R^2p^{\kappa}}{N^3 X^4d}
  \sum_{n_1\asymp \frac{NX^2}{Q^2} } |\lambda_g(n_1)|^2
   \bigg(\sum_{\substack{n_2\asymp \frac{NX^2}{Q^2} \\ n_2\equiv n_1 \bmod dp^{\beta-\kappa+\lambda}}}
  \sum_{\substack{m\ll \frac{dN^{1+\varepsilon}}{R^2p^{\kappa}} \\
  p^{\kappa-\lambda+\beta-\kappa+\lambda}\mid m }}
  p^{2\kappa+\lambda+\delta_1} \\
  & \quad \quad \quad \quad \quad \quad \quad + \sum_{0\leq \ell \leq\lambda-\kappa/2 }
  \sum_{\substack{n_2\asymp \frac{NX^2}{Q^2} \\ n_2\equiv n_1 \bmod dp^{\ell} \\ p^{\ell}\| n_1-n_2}}
  \sum_{\substack{m\ll \frac{dN^{1+\varepsilon}}{R^2p^{\kappa}} \\
   p^{\kappa-\lambda+\ell}\mid m }}
  p^{5\kappa/2+2\ell+\delta_1+\delta_2/2}\bigg),
\end{aligned}
\end{equation}
Plugging \eqref{Sigma_{b1}(d,q)} into \eqref{Sb(N)caseb}, we get
\begin{equation*}
  S^{\flat}_{b1}(N)\ll N^{1/2+\varepsilon}\left(p^{5\kappa/4-\lambda/2+3/4}\frac{K}{t^{1/2}}+p^{\kappa-\lambda/4+1}\frac{K^{5/4}}{t^{1/2}}
  +p^{3\kappa/4+\lambda/4+3/4}\frac{K^{5/4}}{t^{1/2}}\right).
\end{equation*}

\subsection{Case c: $\frac{nN}{q^2 p^{2\lambda}}\ll N^{\varepsilon}$, $\frac{NX}{p^{\lambda}RQ}\ll N^{\varepsilon}$, $\frac{mN}{q^2p^{2\kappa}}\gg N^{\varepsilon}$}

Applying Lemma \ref{voronoiGL2-Maass} to $m$-sum in \eqref{definition of m and n sums} and by  \eqref{The $-$ case}, we are led to estimate
$$
\frac{1}{\tau(\overline{\chi})}\frac{N^{1-it}}{qp^{\kappa}}\sum_{c \mod p^{\kappa}}\overline{\chi}(c)\sum_{m}\lambda_f(m)
e\left(-\frac{\overline{a^{\ast}}m}{qp^{\kappa}}\right)
\Phi_2^{+}\left(\frac{mN}{q^2p^{2\kappa}}\right)
$$
with
\begin{equation}\label{M voronoi asymptotic}
\Phi_2^{+}(y)=y^{-\frac{1}{4}}\int_{0}^{\infty}V(\xi)\xi^{-\frac{1}{4}}
e\left(-\frac{t\log \xi}{2\pi}+ \frac{Nu}{Qqp^{\lambda}}\xi
\pm 2\sqrt{y\xi}\right)\mathrm{d}\xi + O(N^{-A}).
\end{equation}
By repeated integration by parts, $\Phi_2^{+}(y)$ is negligibly small unless $\frac{\sqrt{mN}}{Rp^{\kappa}}\asymp t$, in which case
by making a change of variable $\xi\rightarrow \xi^{2}$ and applying the stationary phase method
we have
$$
\Phi_2^{+}(y)= 2y^{-\frac{1}{4}}t^{-1/2}e\left(-\frac{t}{\pi} \log {\frac{Rtp^{\kappa}}{2\pi e \sqrt{mN}}}\right)W\left(\frac{mN}{R^2p^{2\kappa}t^2}\right) +O(N^{-A}).
$$

Then we apply the $\mathrm{GL_2}$ Voronoi summation formula to $n$-sum and making a change of variable $u \rightarrow Xu$, we have

\begin{equation}
\begin{aligned}
S^{\flat}_c(N)
\ll &\frac{XN^{7/4+\varepsilon}}{p^{2\lambda+\kappa}Q\sqrt{t}}\sum_{\substack{q \sim R\\ (q,p)=1}}\frac{1}{q^{5/2}}
\sideset{}{^\star}\sum_{a\bmod{qp^{\lambda}}}\sum_{c \mod p^{\kappa}}\overline{\chi}(c)\sum_{\eta=\pm 1}\sum_{n}\lambda_g(n)e\left(\eta\frac{\overline{a}n}{qp^{\lambda}}\right)
\sum_{m}\frac{\lambda_f(m)}{m^{1/4}}
 e\left(\frac{-\overline{a^{\ast}}m}{qp^{\kappa}}\right)\\
& \cdot \int_{\mathbb{R}}W_q\left(u\right)g(q,Xu)
\Phi_1^{\sgn{\eta}}\left(\frac{nN}{q^2p^{2\lambda}}\right)
e\left(-\frac{t}{\pi} \log {\frac{Rtp^{\kappa}}{2\pi e \sqrt{mN}}}\right)W\left(\frac{mN}{R^2p^{2\kappa}t^2}\right) \mathrm{d}u.
\end{aligned}
\end{equation}
By Lemma \ref{Phi 1 analysis}(3),
\begin{equation}\label{Phi 1 y small}
\Phi_1^{\sgn{\eta}}(y)=\int_{0}^{\infty}W(\xi)
e\left(- \frac{NXu\xi}{Qqp^{\lambda}}\right)J_g^{\sgn{\eta}}(y\xi)\mathrm{d}\xi,
\end{equation}
with $y^k {J_g^{\sgn{\eta}}}^{(k)}(y) \ll_{k, g} N^{\varepsilon}.$

Applying the Cauchy--Schwarz inequality, we get
\begin{equation}\label{Sb(N)case c}
\begin{aligned}
S^{\flat}_c(N)\ll \frac{XN^{3/2}R^{1/2}}{p^{2\lambda+\kappa/2}Q}\sum_{\substack{q \sim R \\ (q,p)=1}}\frac{1}{q^{5/2}} \sum_{d|q}d \cdot \Sigma_c(d,q)^{1/2},
\end{aligned}
\end{equation}
where
\[
\Sigma_c(d,q) :=
\sum_{m}W_2\left(\frac{m}{\frac{R^2p^{2\kappa}t^2}{N}}\right)\\
   \bigg|\sum_{\substack{n\ll \frac{q^2p^{2\lambda}N^{\varepsilon}}{N} \\ n\equiv mp^{\lambda}\overline{p^{2\kappa-\lambda}}\bmod d}}\lambda_g(n)
  \mathfrak{C}^{\ast}(m,n,q)J(m,n) \bigg|^2,
\]
with
$$
\mathfrak{C}^{\ast}(m,n,q)= \sum_{c \mod p^{\kappa}}\overline{\chi}(c)
\sideset{}{^\star}\sum_{b\bmod{p^{\lambda}}}
e\left(\frac{-\overline{a^{\ast}}m\bar{q}}{p^{\kappa}}+ \frac{\eta n\bar{q}\bar{b}}{p^{\lambda}}\right)
$$
and
$$
J(m,n)= \int_{\mathbb{R}} W_q\left(u\right)g(q,Xu) e\left(-\frac{t}{\pi} \log {\frac{Rtp^{\kappa}}{2\pi e \sqrt{mN}}} \right)
\Phi_1^{\sgn{\eta}}\left(\frac{nN}{q^2p^{2\lambda}}\right)W\left(\frac{mN}{R^2p^{2\kappa}t^2}\right)  \mathrm{d}u.
$$

We apply the Poisson summation formula on $m$ with modulo $dp^{\kappa}$ and get
\begin{equation*}
  \Sigma_c(d,q) \ll \frac{R^2p^{\kappa}t^2}{Nd}
  \sum_{n_1\ll \frac{q^2p^{2\lambda}N^{\varepsilon}}{N}}  |\lambda_g(n_1)|^2
  \sum_{\substack{n_2\ll \frac{q^2p^{2\lambda}N^{\varepsilon}}{N}  \\ n_2\equiv n_1 \bmod d}}\sum_{m\in\mathbb{Z}}|\mathfrak{C}(m,n_1,n_2)\mathfrak{I}_c(m,n_1,n_2)|,
\end{equation*}
where
$$
  \mathfrak{C}(m,n_1,n_2) = \sum_{ \gamma \bmod p^{\kappa}}\mathfrak{C}^{\ast}(\gamma,n_1,q)\overline{\mathfrak{C}^{\ast}(\gamma,n_2,q)}
  e\left(\frac{m\left(n_1p^{3\kappa-\lambda}\overline{p^{\kappa+\lambda}}+\gamma d\bar{d}\right)}{dp^{\kappa}}\right)
$$
and
\begin{equation}
  \mathfrak{I}_c(m,n_1,n_2) = \int_{\mathbb{R}} W_2\left( \xi\right)
  J\left(\frac{R^2p^{2\kappa}t^2\xi}{N}, n_1\right)\overline{J\left(\frac{R^2p^{2\kappa}t^2\xi}{N}, n_2\right)}
  e\left(-\frac{mR^2p^{\kappa}t^2}{Nd} \xi \right)  \dd \xi.
\end{equation}

\begin{lemma}
\begin{enumerate}
\item If $m \gg \frac{N^{1+\varepsilon}d}{R^2p^{\kappa}t^2}$, we have
$\mathfrak{I}_c(m,n_1,n_2)\ll N^{-A}$.
\item If $m \ll \frac{N^{1+\varepsilon}d}{R^2p^{\kappa}t^2}$, we have
$\mathfrak{I}_c(m,n_1,n_2)\ll 1$.
\end{enumerate}
\end{lemma}
\begin{proof}
The first result can be done by applying integration by parts with respect to the $\xi$-integral. For $m\ll \frac{N^{1+\varepsilon}d}{R^2p^{\kappa}t^2}$, we have the trivial bound $\mathfrak{I}_c(m,n_1,n_2)\ll 1$.
\end{proof}
The contribution from the zero frequency to $\Sigma_c(d,q)$ is
\begin{equation}\label{Sigma_{c0}(d,q)}
  \Sigma_{c0}(d,q) \ll \frac{R^2p^{\kappa}t^2}{Nd}
  \sum_{n_1 \ll \frac{q^2p^{2\lambda}N^{\varepsilon}}{N}} |\lambda_g(n_1)|^2
  \left(\sum_{\substack{n_2 \ll \frac{q^2p^{2\lambda}N^{\varepsilon}}{N} \\ n_2\equiv n_1 \bmod dp^{\lambda}}}   p^{2\kappa+\lambda} + \sum_{\substack{n_2 \ll \frac{q^2p^{2\lambda}N^{\varepsilon}}{N} \\ n_2\equiv n_1 \bmod dp^{\lambda-1}}}   p^{2\kappa+\lambda-1}\right).
\end{equation}
Plugging \eqref{Sigma_{c0}(d,q)} into \eqref{Sb(N)case c}, we get
\begin{equation*}
  S^{\flat}_{c0}(N)\ll N^{1/2+\varepsilon}\left(p^{3\kappa/2-3\lambda/4}\frac{t^{3/2}}{K^{5/4}}
  +p^{\kappa-\lambda/2}\frac{t}{K^{3/2}}\right).
\end{equation*}
The contribution from the non-zero frequencies to $\Sigma_c(d,q)$ is
\begin{equation}\label{Sigma_{c1}(d,q)}
\begin{aligned}
  \Sigma_{c1}(d,q)\ll  &\frac{R^2p^{\kappa}t^2}{Nd}
  \sum_{n_1 \ll \frac{q^2p^{2\lambda}N^{\varepsilon}}{N} } |\lambda_g(n_1)|^2
   \bigg(\sum_{\substack{n_2\ll \frac{q^2p^{2\lambda}N^{\varepsilon}}{N} \\ n_2\equiv n_1 \bmod dp^{\beta-\kappa+\lambda}}}
  \sum_{\substack{m \ll \frac{N^{1+\varepsilon}d}{R^2p^{\kappa}t^2} \\
  p^{\kappa-\lambda+\beta-\kappa+\lambda}\mid m }}
  p^{2\kappa+\lambda+\delta_1}\\
  & \quad \quad \quad \quad \quad \quad \quad \quad \quad \quad + \sum_{0\leq \ell \leq\lambda-\kappa/2 }
  \sum_{\substack{n_2\ll \frac{q^2p^{2\lambda}N^{\varepsilon}}{N} \\ n_2\equiv n_1 \bmod dp^{\ell}\\
  p^{\ell}\| n_1-n_2}}
  \sum_{\substack{m \ll \frac{N^{1+\varepsilon}d}{R^2p^{\kappa}t^2}\\ p^{\kappa-\lambda+\ell}\mid m }}
  p^{5\kappa/2+2\ell+\delta_1+\delta_2/2}\bigg).
\end{aligned}
\end{equation}
Plugging \eqref{Sigma_{c1}(d,q)} into \eqref{Sb(N)case c}, we get
\begin{equation*}
  S^{\flat}_{c1}(N)\ll N^{1/2+\varepsilon}\left(p^{5\kappa/4-\lambda/2+3/4}\frac{t}{K}+p^{\kappa-\lambda/4+1}\frac{t^{1/2}}{K^{5/4}}
  +p^{3\kappa/4+\lambda/4+3/4}\frac{t^{1/2}}{K^{5/4}}\right).
\end{equation*}

\end{document}